\input ./style/arxiv-general.cfg
\documentclass[aap,MSNbibl,seceqn,dvips]{arximspdf}
\makeatletter
   \@ifpackageloaded{graphicx}{}{\usepackage{graphicx}}
\makeatother

\doi{10.1214/14-AAP1051} 
\volume{25}
\issue{5}
\pubyear{2015}
\firstpage{2383}
\lastpage{2415}
\docsubty{FLA}

\makeatletter

\newcommand{\rrvert}{\vert}
\newcommand{\llvert}{\vert}
\newcommand{\eqref}[1]{(\ref{#1})}

\def\ii{\mathtt{i}}

\newcommand{\te}{\widehat{e}}

\newtheorem{Thm}{Theorem}[section]
\newtheorem{Cor}[Thm]{Corollary}
\newtheorem{Lemma}[Thm]{Lemma}

\newtheorem{Lemm}{Lemma}[section]
\newtheorem{Prop}[Thm]{Proposition}
\newtheorem{Propp}[Lemm]{Proposition}

\newproclaim{Def}[Thm]{Definition}
\newproclaim{Ex}[Thm]{Example}
\newproclaim{As}[Thm]{Assumption}
\newproclaim{Rem}[Thm]{Remark}
\newproclaim{Remm}[Lemm]{Remark}
\newproclaim{Rems}{Remarks}

\def\mbb{\mathbb}
\def\unl{\underline}

\renewcommand{\te}[1]{\mathrm{e}^{#1}}
\def\WH{\widehat}

\def\a{\alpha}
\def\th{\theta}
\def\e{\varepsilon}
\def\t{\tau}

\def\td{\mathrm{d}}

\makeatother

\begin{document}
\begin{frontmatter}


\title{Explicit solution of an inverse first-passage time problem for
L\'{e}vy processes and counterparty credit risk}
\runtitle{Inverse first-passage time problem for
L\'{e}vy processes}

\begin{aug}
\author[A]{\fnms{M.~H.~A.}~\snm{Davis}\ead[label=e1]{mark.davis@imperial.ac.uk}}
\and
\author[A]{\fnms{M.~R.}~\snm{Pistorius}\corref{}\ead[label=e2]{m.pistorius@imperial.ac.uk}}
\runauthor{M.~H.~A. Davis and M.~R. Pistorius}
\affiliation{Imperial College London}
\address[A]{Department of Mathematics\\
Imperial College London\\
South Kensington Campus\\
London SW7 2AZ\\
United Kingdom\\
\printead{e1}\\
\phantom{E-mail:\ }\printead*{e2}}
\end{aug}

\received{\smonth{6} \syear{2013}}
\revised{\smonth{3} \syear{2014}}


\begin{abstract}
For a given Markov process $X$ and survival function
$\overline H$ on $\mbb R^+$,
the \emph{inverse first-passage time problem} (IFPT) is to find a
barrier function $b\dvtx\mbb R^+\to[-\infty,+\infty]$
such that the survival function of the first-passage time $\tau_b=\inf
\{
t\ge0\dvtx X(t) < b(t)\}$ is
given by $\overline H$. In this paper, we consider a version of the IFPT
problem where the barrier is
\emph{fixed at zero} and the problem is to find an initial distribution
$\mu$ and a time-change $I$ such that
for the time-changed process $X\circ I$ the IFPT problem is solved by a
constant barrier at the level zero.
For any L\'{e}vy process $X$ satisfying an exponential moment condition,
we derive the solution of this problem in terms of {\em$\lambda
$-invariant distributions} of the process $X$ killed at the epoch of
first entrance into the negative half-axis. We provide an explicit
characterization of such distributions,
which is a result of independent interest. For a given multi-variate
survival function $\overline H$ of generalized frailty type, we construct
subsequently an explicit solution to the corresponding IFPT with the
barrier level fixed at zero.
We apply these results to the valuation of financial contracts that are
subject to counterparty credit risk.
\end{abstract}
%
%

%
\end{frontmatter}

\section{Introduction}\label{sec:intro}
Financial models incorporating the idea that a firm defaults on its
debt when the value of the debt exceeds the value of the
firm were originally introduced by Merton \cite{Mer74}. Black and
Cox~\cite{BC} extended the Merton model
by modelling the time of default as the first time that the value of
the firm less the value of its debt becomes negative. Because ``firm
value'' cannot be directly measured, later contributors such as
Longstaff and Schwartz~\cite{LonSch95} and Hull and White \cite{HW}
have moved to stylized models in which default occurs when some process
$Y(t)$---interpreted as ``distance to default''---crosses a given, generally
time-varying, barrier $b(t)$. The risk-neutral
distribution of the default time can be inferred from the firm's credit
default swap spreads, and Hull and White~\cite{HW} provide a
numerical algorithm to determine $b(t)$ such that the first hitting
time distribution $H$ is equal to this market-implied
default time distribution when $Y(t)$ is Brownian motion.

As we will show, these calculations are greatly simplified if, instead
of starting at a fixed point $Y(0)=x>0$ and calibrating the barrier
$b(t)$ we fix the barrier at $b(t)\equiv0$ and start $Y$ at a random
point $Y(0)=Y_0$, where $Y_0$ has a distribution function $F$ on
$\mathbb{R}^+$, to be
chosen. If we combine this with a deterministic time change then it
turns out that essentially any continuous
distribution $H$ can be realized in this way, often with closed-form
expressions for $F$.

In precise terms, the \emph{inverse first-passage time} (IFPT) problem
may be described as follows. Let $(Y,P^\mu)$ be a real-valued Markov
process with
c\`adl\`ag\footnote{c\`adl\`ag${} = {}$right-continuous with left-limits.} paths
that has initial distribution $\mu$ on $\mbb R^+\setminus\{0\}$
[i.e., $P^\mu(Y_0\in\td x) = \mu(\td x)$]. Given a CDF $H$
on $\mbb R^+$,
the IFPT for the process $(Y,P^\mu)$ is to find a barrier function
$b\dvtx\mbb R^+\to[-\infty,+\infty]$ such that the first-passage time
$\tau
^Y_b$ of the process $Y$ below the barrier $b$ has CDF $H$:
%
\begin{eqnarray}
\label{eq:PMt}  P^\mu\bigl(\tau^Y_b\leq t
\bigr)& =& H(t),\qquad t\in\mbb R^+,\mbox{ with }
\nonumber
\\[-8pt]
\\[-8pt]
\nonumber
\tau^Y_b& =& \inf\bigl\{t\in\mbb R^+\dvtx
Y_t \in \bigl(-\infty, b(t) \bigr)\bigr\}.
\end{eqnarray}
Recently, there has been a renewed interest in the IFPT problem, in
good part motivated by the above questions of credit risk modelling.
Chen {et al.}~\cite{Chadam} prove existence and uniqueness of the
IFPT of an arbitrary continuous CDF on $\mbb R^+$ for a diffusion with
smooth bounded coefficients and strictly positive volatility function. In
\cite{AZ,HT,HW,SongZipkin,ZucSac}, a number of methods have been developed
to compute this boundary, which is in general nonlinear. Zucca and
Sacerdote~\cite{ZucSac} analyse a Monte Carlo
approximation method and a method based on the discretization of the
Volterra integral equation satisfied
by the boundary, which was derived in Peskir~\cite{Peskir2}, while
related integral equations are studied in
Jaimungal {et al.}~\cite{Seb}. Avellaneda and Zhu~\cite{AZ} derive
a free boundary problem for the density
of a diffusion killed upon first hitting the boundary, where the free
boundary is the solution to the IFPT, and
Cheng {et al.}~\cite{Chadam2} established the existence and
uniqueness of a solution to this
free-boundary problem.
A related ``smoothed'' version of the IFPT problem is considered in
Ettinger {et al.}~\cite{EEH}:
for any prescribed life-time it is shown that there exists a unique
continuously differentiable boundary
for which a standard Brownian motion killed at a rate
that is a given function of this boundary has the prescribed life-time.

In this paper, we consider a related inverse problem where the barrier
is \emph{fixed} to be equal to zero,
and the problem is to identify in a given family a stochastic process
whose first-passage time below the level zero has the given probability
distribution.
For a given Markov process $X$, the class of stochastic processes that
we consider consists
of the collection $(P^\mu, X\circ I)$ that is obtained by time-changing
$X$ by a continuous increasing function $I$
and by varying the initial distribution $\mu$ of $X$
over the set of all probability measures on the positive half-line. Here,
$I\dvtx\mbb R^+\to[0,\infty]$ is a function that is continuous
and increasing on its domain, that is, at all $t$ for which
$I(t)$ is finite, and
the time-changed process $X\circ I =\{(X\circ I)(t), t\in\mbb R^+\}$ is
defined by
$(X\circ I)(t) = X(I(t))$ if $I(t)$ is finite, and by $\limsup_{t\to
\infty}X(t)$ otherwise.

\begin{Def} For a continuous CDF $H$ on $\mbb R^+$, the \emph
{randomized and time-changed inverse first-passage problem}
(RIFPT) is to find a
probability measure $\mu$ on $(\mbb R^+, \mathcal B(\mbb R^+))$ and an
increasing continuous function $I\dvtx\mbb R^+\to[0,\infty]$
such that for the time-changed process $Y = X\circ I$
the first-passage time into the negative half-line $(-\infty,0)$ has
CDF $H$:
%
\begin{eqnarray}
\label{eq:TFPT0}  P^{\mu}\bigl(\tau^Y_0 \leq t
\bigr) &=& H(t),\qquad t\in\mbb R^+,\mbox{ with}
\nonumber
\\[-8pt]
\\[-8pt]
\nonumber
 \tau^Y_0 &= &\inf\bigl\{t\in\mbb R^+\dvtx
Y_t \in(-\infty,0) \bigr\} .
\end{eqnarray}
\end{Def}

The fact that the boundary is constant and known is helpful for
practical implementation of the model, for example, in subsequent
counterparty risk valuation computations and for the matching of model
and market prices.

In this paper, we concentrate on the case where $X$ is a L\'evy process
satisfying an exponential moment condition. The class of L\'{e}vy
processes has been extensively deployed in financial modelling; see
the monograph Cont and Tankov~\cite{CT}. For the general theory of L\'
{e}vy processes, we refer to the monographs Applebaum \cite{Apple},
Bertoin~\cite{Bertoin}, Kyprianou~\cite{Kyprianou}
and Sato~\cite{Sato}.

The key step is to
determine, for some $\lambda\in\mathbb{R}_+$, a $\lambda$-invariant
distribution for the process $X$ killed at the first hitting time of 0,
which is a result of independent interest; see Definition~\ref
{def:linv} below. If $\mu$ is $\lambda$-invariant then under $P^{\mu}$,
the first-passage time $\tau^X_0$ is exponentially distributed with
parameter $\lambda$, so $(\mu,I)$ with $I(t)=t$ solves the RIFPT
problem when $H$ is $\mathrm{Exp}(\lambda)$. The solution for other
continuous distribution functions $H$ is then obtained by an obvious
deterministic time change.

The paper is structured as follows. In Section~\ref{sec:prob}, we
formulate the problem and state the main results for the RIFPT problem,
Theorems \ref{thm:I} and \ref{thm:II}. The proof of Theorem~\ref{thm:I}
is also given, together with an illustrative example where the L\'evy
process is drifting Brownian motion. In Section~\ref{sec:md-rifpt}, a
multi-dimensional version of the RIFPT theorem is stated for a specific
class of multivariate default-time distributions; its proof follows
quite easily given the results of Section~\ref{sec:prob}. The proof of
Theorem~\ref{thm:II}, which is presented in Section~\ref{sec:pfthmii},
involves the relationship between first-passage times and the so-called
Wiener--Hopf factors; these matters are discussed in Section~\ref
{sec:prel0}. In Section~\ref{sec:pthecase}, the results of Theorem~\ref
{thm:II} are illustrated explicitly for the special case of \emph
{mixed-exponential L\'evy processes} (a self-contained proof of the
quasi-invariance in this case, based on residue calculus, is given in the
\hyperref[ssec:resd]{Appendix}).
The concluding Section~\ref{sec:crv} demonstrates the application
of our results to a problem of counterparty risk valuation.

\section{IFPT problem formulation and main results}\label{sec:prob}
Let $(\Omega, \mathcal F, \mathbf F, P)$ be a filtered probability
space with
completed filtration $\mathbf F=\{\mathcal F_t\}_{t\ge0}$, and $X$
be an $\mathbf
F$-L\'evy process, that is, an $\mathbf F$-adapted stochastic process with
c\`adl\`ag paths that has stationary independent increments, with
$X_0=0$ and the property that for each $s\leq t<u$ the increment
$X_u-X_t$ is independent of $\mathcal F_s$.
Let $\{P_x, x\in\mbb R\}$ be the
family of probability measures corresponding to shifts of the
L\'{e}vy process $X$ by $x$ and, more generally, denote by $P^\mu$ the
family of measures with initial distribution (the distribution of
$X_0$) equal to $\mu$; thus $P_x=P^{\delta_x}$ where $\delta_x$ is the
Dirac measure at $x$ and $P=P_0$. To avoid degeneracies, we exclude
throughout the case that $X$ has monotone paths. As standing notation,
we denote
$X_*(t)=\inf_{s\leq t}X(s)$ and $X^*(t)=\sup_{s\leq t}X(s)$.
Below we describe a solution to the RIFPT problem under the following
conditions.

\begin{As}\label{as:smooth} The Gaussian coefficient $\sigma^2$ and
L\'{e}vy measure $\nu$ of $X$ satisfy
at least one of the following conditions:
\begin{eqnarray*}
\mbox{(i)}&&\quad \sigma^2>0,\qquad \mbox{(ii)}\quad \nu(-1,1) = +\infty,\qquad
\\
\mbox{(iii)}&&\quad \mbox{$\nu$ has no atoms and $S^\nu\cap(-\infty ,0)\neq
\varnothing$,}
\end{eqnarray*}
where $S^\nu$ denotes the support of $\nu$.
\end{As}

When only Assumptions \ref{as:smooth}(iii) holds, the process $X$ is of
the form
$X_t = \mathtt d t + \sum_{s\in(0,t]}\Delta X_s$,
where $\Delta X_s = X_s - X_{s^-}$ denotes the jump-size of $X$ at time $s$,
for some constant $\mathtt d$, which is called the infinitesimal drift
of $X$.

The first observation is that for any initial distribution there exists
a unique time-change that solves the RIFPT problem.
For a given probability measure $\mu$ on the positive real line,
define the function $I_\mu\dvtx\mbb R^+\to[0,\infty]$ by
%
\begin{eqnarray}
\label{eq:Imu}  I_\mu(t) &=& \overline F_\mu^{-1}
\bigl(\overline H(t)\bigr), \qquad t\in\mbb R^+,\mbox{ with}
\\
 \overline F_\mu^{-1}(x) &=& \inf\bigl\{t\in
\mbb R^+\dvtx \overline F_\mu (t)<x\bigr\},
\end{eqnarray}
where $\overline H = 1 - H$ and $\overline F_\mu$ denote
the survival functions corresponding to the CDF $H$ and to the CDF
of the first-passage time $\tau^X_0$ of $X$ into the
negative half-line $(-\infty,0)$ under the probability measure $P^\mu$.
Here and throughout this paper, we use the convention $\inf\varnothing
=+\infty$.

\begin{Thm}\label{thm:I}
Let $H$ be a continuous CDF on $\mbb R^{+}$,
and let $\mu$ be a probability measure on
$(\mbb R^{+},\mathcal B(\mbb R^{+}))$ with $\mu(\{0\})=0$.
Assume Assumption~\ref{as:smooth} holds and that $\mu$ has no atoms if
only Assumption~\ref{as:smooth}\textup{(iii)} is satisfied. Then the function
$I_\mu$ defined in \eqref{eq:Imu}
is the unique time-change such that $(\mu, I_\mu)$ is a solution of the
RIFPT problem.
\end{Thm}

For the proof, we need some properties of the distribution of the
running infimum.

\begin{Lemma}\label{lem:distinf}
\emph{(i)} If $X$ satisfies Assumption~\ref{as:smooth}\textup{(i)}
or \textup{(ii)},
the CDF of $X_*(t)$ is continuous, for any $t>0$.

\emph{(ii)} Alternatively, if only Assumption~\ref
{as:smooth}\textup{(iii)} holds, then for any $t>0$
the CDF of $X_*(t)$ is continuous on the set $\mbb R_-\setminus\min\{
\mathtt d t,0\}$,
with $\mbb R_-=(-\infty,0]$.
\end{Lemma}

The proof of Lemma~\ref{lem:distinf}(i) can be found in Sato~\cite{Sato},
Lemma~49.3, and
Pe{\v{c}}erski{\u\i} and Rogozin~\cite{PR}, Lemma~1, while Lemma~\ref
{lem:distinf}(ii) follows by conditioning
on the first jump of the process $X$.

\begin{pf*}{Proof of Theorem~\ref{thm:I}}
Denote by $\mathtt c$ the value 0 or $\max\{-\mathtt d,0\}$ according
to whether or not
$X$ satisfies at least one of the Assumptions \ref{as:smooth}(i) and (ii).
The key observation in the proof is that for any $x>0$
the map $t\mapsto P_x(\tau_0^X>t)$ is (a) strictly decreasing
and (b) continuous at any $t$ satisfying $\mathtt c t\neq x$.
To verify claim (b), it suffices to show that
$P_x(\tau_0^X=t)$ is zero for any nonnegative $t$ that is such that
$\mathtt c t\neq x$.
The latter follows as consequence of the bound $P_x(\tau_0^X=t) \leq
P_0(X_*(t)=-x)$ that holds
for any strictly positive $x$ and $t$, and the fact (from Lemma~\ref
{lem:distinf})
that the CDF of $X_*(t)$ is continuous on $(-\infty,0]\setminus\{
-\mathtt c t\}$.
To see that claim (a) is true, we observe that, by the Markov property,
we have for strictly positive $x$, $t$ and $s$,
%
\begin{eqnarray}
\label{eq:lb2}
&&P_x \bigl(\tau^X_0>t
\bigr) - P_x \bigl(\tau^X_0>t+s \bigr)\nonumber \\
&&\qquad=
P_x \bigl(\tau ^X_0 > t,
\tau_0^X \leq t+s \bigr)
\\
&&\qquad\ge E \bigl(\mathbf 1_{\{X_*(t) > -x\}} P\bigl(X_*(s) < -x-z\bigr)|_{z=X_t}
\bigr).\nonumber
\end{eqnarray}
Since for any strictly positive epoch $s$ the random variable $X_s$ has
an infinitely divisible distribution
and the support of an infinitely divisible
distribution not corresponding to the sum of a subordinator and a
deterministic drift
is unbounded from below
(e.g., \cite{Sato}, Corollary~24.4),
it follows that under Assumptions \ref{as:smooth} we have
%
\begin{equation}
\label{eq:lb1} P\bigl(X_*(s) < -x\bigr) \ge P(X_s < -x) > 0,\qquad  s>0, x
\ge0.
\end{equation}
By combining~\eqref{eq:lb2} and~\eqref{eq:lb1},
we have for any strictly positive $x$, $t$ and $s$,
\[
P_x \bigl(\tau_0^X > t \bigr) >
P_x \bigl(\tau_0^X > t+s \bigr),
\]
and hence (b) holds true.

The above key observation in conjunction with Lebesgue's dominated
convergence theorem and the assumption that $\mu$ has no atoms
if $X$ does not satisfy Assumption~\ref{as:smooth}(i) and (ii) imply
that the map
$t\mapsto\overline F_\mu(t)$ is continuous and strictly decreasing.
Denote by $Y^\mu$ the time-changed process $X\circ I_\mu$.
Since $I_\mu$ is monotone increasing and continuous, we have
%
\begin{equation}
P^{\mu} \bigl(\tau^{Y^{\mu}}_0 \ge t \bigr) =
P^{\mu} \bigl(\tau^X_0\ge I_\mu(t)
\bigr) = \overline F_\mu \bigl(\overline F_\mu^{-1}
\bigl(\overline H(t)\bigr) \bigr) = \overline H(t)
\end{equation}
for $t\in\mbb R^+$, where we used in the final equality that
$\overline F_\mu$ is continuous.
\end{pf*}

We next turn to the specification of the second degree of freedom, the
initial distribution $\mu$.
By an appropriate choice of the randomisation $\mu$ the form of the function
$\overline F_\mu$ in the specification of the time-change $I_\mu$ in
\eqref{eq:Imu}
can be considerably simplified. In particular, the function $\overline
F_\mu
$ is equal to an exponential
if $\mu$ is taken to be equal to any quasi-invariant distribution
of the process $X$ killed
at the epoch of first-passage below the level~0, the definition of
which, we recall, is as follows.

\begin{Def}\label{def:linv}
For given $\lambda\in\mbb R^+$, the probability measure $\mu$ on the
measurable
space $(\mbb R^+, \mathcal B(\mbb R^+))$ is a {\em$\lambda$-invariant
distribution} for the process $X$ killed
at the epoch of first entrance into the negative half-axis $(-\infty
,0)$ if
%
\begin{equation}
\label{eq:linv} P^\mu \bigl(X_t\in A, t <
\tau_0^X \bigr) = \mu(A) \te{-\lambda t} \qquad \mbox{for all
$A\in\mathcal B\bigl(\mbb R^+\bigr)$.}
\end{equation}
The probability measure $\mu$ is a \emph{quasi-invariant distribution of
$\{X_t, t<\tau^X_0\}$}
if $\mu$ is a $\lambda$-invariant distribution of $\{X_t, t<\tau
^X_0\}$
for some $\lambda\in\mbb R^+$.
\end{Def}

To guarantee existence of quasi-invariant distributions, we restrict ourselves
in the subsequent analysis to L\'{e}vy processes $X$ that admit a
finite exponential moment.

\begin{As}\label{A1}
The distribution of $X_1$ satisfies the following exponential moment condition:
\[
E\bigl[\te{\e X_1}\bigr]<1\qquad \mbox{for some $\e\in(0,\infty)$},
\]
where $E[\cdot]$ denotes the expectation under
the probability measure $P(=P_0)$.
\end{As}

\begin{figure}

\includegraphics{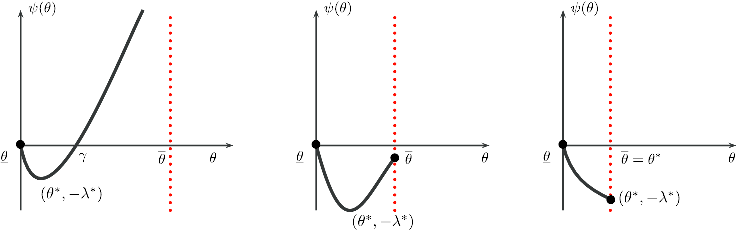}

\caption{Three graphs of Laplace
exponents $\psi$ of L\'{e}vy processes satisfying Assumption \protect
\ref{A1},
with $-\lambda^*= \min_{\theta\in[\unl\theta,\overline \theta
]}\psi
(\theta) =
\psi(\theta^*)$,
where $[\unl\theta,\overline \theta]$ denotes the closure of the
domain of
$\psi$.
In the left-hand figure, $\gamma$ denotes the largest root of the
Cram\'
{e}r--Lundberg equation $\psi(\theta)=0$
and $\theta^*<\overline \theta$ satisfies the equation $\psi
'(\theta
)=0$. In
the right-hand figure,
$\theta^*$ and $\overline \theta$ coincide.}\label{fig:psi}
\end{figure}

Under Assumption~\ref{A1}, there exists a continuum of quasi-invariant
distributions of the process $X$ killed upon the first moment of
entrance into the negative half-axis, which are given in terms of the
Laplace exponent and the positive Wiener--Hopf factor of $X$.

The Laplace exponent $\psi\dvtx\mbb R\to(-\infty,\infty]$ of $X$,
given by $\psi(\theta)= \log E[\te{\theta X_1}]$ for real $\theta$, is
finite valued and convex
when restricted to the interior $(\underline\theta,\overline\theta
)$ of
its maximal domain,
where $\overline\theta=\sup\{\th\in\mbb R\dvtx E[\exp\{\th X_1\}
]<\infty\}$ and
$\underline{\theta} = \inf\{\th\in\mbb R\dvtx E[\exp\{\th X_1\}
]<\infty\}$
(see Figure~\ref{fig:psi} for plots
of Laplace exponents of L\'{e}vy processes satisfying Assumption~\ref{A1}.)
Since $\psi$ is a convex lower-semi-continuous function that under
Assumption~\ref{A1} takes a strictly negative value at some
$\varepsilon
>0$, it follows that the infimum of $\psi$ is strictly negative and is
attained at some $\theta^*\in[\underline\theta,\overline\theta]$,
that is,
%
\begin{equation}
\label{eq:petrov} -\lambda^*:= \inf_{\theta\in[\underline{\theta},\overline
\theta
]}\psi(\theta) = \psi
\bigl(\theta^*\bigr)<0.
\end{equation}
On the interval $(\underline\theta,\theta^*]$
the function $\psi$ is continuous and
strictly monotone decreasing with inverse denoted by
%
\begin{equation}
\label{eq:phiphi} \bar\phi: \bigl[-\lambda^*,\psi(\underline\theta)\bigr ) \to\bigl (
\underline \theta,\theta^* \bigr].
\end{equation}
In particular, we note $\psi'(0+)\in[-\infty,0)$ so that the mean
$E[X_1]$ of $X_1$ is strictly negative.

The positive Wiener--Hopf factor is the function
$\Psi^+\dvtx(0,\infty)\times\mbb D^+\to\mbb C$ with $\mbb D^+:=\{
u\in
\mbb
C\dvtx\Im(u)\ge0\}$ given by
%
\begin{equation}
\label{eq:Psi+def} \Psi^+(q,\theta)= E\bigl[\exp\bigl(\ii\theta X^*_{e(q)}
\bigr)\bigr],\qquad  q>0, \th\in\mbb D^+,
\end{equation}
with
$e(q)$ denoting an Exp($q$) random time
that is independent of $X$. In Lemma~\ref{lem:anex}, we show that
the function $\Psi^+$ can be uniquely extended to the set
$\{(q,\theta)\dvtx\Re(q) \ge-\lambda^*, \Im(\theta)\ge-\theta
^*\}
\setminus\{(-\lambda^*,-\theta^*)\}$ (by analytical continuation and
continuous extension); this extension is also denoted by $\Psi^+$.

Consider for any $\lambda\in(0,\lambda^*]$
the function $\WH\mu_\lambda\dvtx\mbb R^+\to\mbb C$ given by
%
\begin{equation}
\label{eq:mul} \WH\mu_\lambda(\theta) = \frac{\bar\phi(-\lambda)}%
{\bar\phi(-\lambda) + \theta}\cdot\Psi^+(-
\lambda, \ii\theta),
\end{equation}
where $\bar\phi$ denotes the inverse of the Laplace exponent as
described above.
The function $\WH\mu_\lambda$ is the Laplace transform
of some probability measure $\mu_\lambda$---an explicit expression for
$\mu_\lambda$ is given in
Lemma~\ref{lem:LT}.
The members of the family $\{\mu_\lambda, \lambda\in(0,\lambda^*]\}$
are quasi-invariant distributions of $\{X_t, t< \tau^X_0\}$:

\begin{Thm}\label{thm:II}
Assume that $X$ is a L\'{e}vy process satisfying\break $E[\exp(-\e X_1)]<1$
for some $\e\in(0,\infty)$.
Then, for any $\lambda\in(0,\lambda^*]$, $\WH\mu_\lambda$ is the
Laplace transform of some probability measure $\mu_{\lambda}$ on
$(\mbb
R^+,\mathcal B(\mbb R^+))$,
which is the unique $\lambda$-invariant distribution of $\{X_t, t<
\tau
^X_0\}$,
the process $X$ that is killed upon the epoch of first-passage into the
negative half-line
$(-\infty,0)$.
\end{Thm}

In the case that $X$ is a mixed-exponential L\'{e}vy process,
the measures $\mu_\lambda$, $\lambda\in(0,\lambda^*]$,
can be shown to be equal to certain
mixed-exponential distributions; see Sections~\ref{sec:pthecase}.

Under any of the initial distributions $\mu_\lambda$ given in Theorem~\ref{thm:II},
the distribution of the first-passage time $\tau^X_0$ is exponential,
and thus the corresponding survival function $\overline F_{\mu
_\lambda}$ and
time change $I_{\mu_\lambda}$ defined in \eqref{eq:Imu}
take explicit forms:
\begin{eqnarray*}
\overline F_{\mu_\lambda}(t) &=& \exp(-\lambda t), \qquad t\in\mbb R^+,
 \lambda\in\bigl(0,
\lambda^*\bigr],
\\ I_{\mu_\lambda}(t) &=& - \frac{1}{\lambda} \log\overline H(t).
\end{eqnarray*}
When the survival function $\overline H$ is continuous, $I_{\mu
_\lambda}(t)$
is equal to a
multiple of the cumulative hazard rate integrated over the interval $[0,t]$.

The combination of Theorems \ref{thm:I} and \ref{thm:II} immediately
yields the following result.

\begin{Cor}\label{thm:IFPT} For any given continuous survival function
$\overline H$ and
$\lambda\in(0,\lambda^*]$, the RIFPT problem is solved by the pair
$(\mu_\lambda, I_{\mu_\lambda})$, that is,
\[
P^{\mu_\lambda} \bigl(\tau^{Y^{\mu_\lambda}}_0 > t \bigr) = \overline
H(t),\qquad  t\in\mbb R^+.
\]
\end{Cor}

\subsection{Example} As a simple example, let us consider the case
where $X_t$ is Brownian motion with drift, with initial distribution
$\mu$, or equivalently $ X_t=X_0+W_t+\eta t$
where $W_t$ is a standard Brownian motion, $\eta\in\mathbb{R}$ and
$X_0\sim\mu$ is a random variable independent of $\{W_t, t\in\mathbb
{R}^+\}$. In this case,
\[
\psi(\theta)=\log E\bigl[\te{\theta X_1}\bigr]=\eta\theta+\tfrac{1}{2}
\theta^2
\]
and $\underline{\theta}=-\infty, \overline{\theta}=+\infty$, so the
coefficients in \eqref{eq:petrov} are
$ \theta^*=-\eta, \lambda^*=\frac{1}{2}\eta^2$
and the inverse of $\psi$ to the left of $\theta^*$ is
\[
\overline \phi(y)=-\eta-\sqrt{\eta^2+2y}.
\]
The positive Wiener--Hopf factor is
\[
\Psi^+(q,\theta)=\frac{-\ii(\eta-\sqrt{\eta^2+2q})}{\theta-\ii
(\eta
-\sqrt{\eta^2+2q})}.
\]
The Laplace transform of the $\lambda$-invariant distribution is
therefore given by
%
\begin{eqnarray}\label{ex}
\widehat{\mu}_\lambda(\theta)&=& \biggl(\frac{-\eta
-\sqrt
{\eta
^2-2\lambda}}{\theta-(\eta+\sqrt{\eta^2-2\lambda})} \biggr) \biggl(
\frac{-\eta+\sqrt{\eta^2-2\lambda}}{\theta-(\eta-\sqrt
{\eta
^2-2\lambda})} \biggr)
\nonumber
\\[-8pt]
\\[-8pt]
\nonumber
&=&\frac{2\lambda}{\theta_+-\theta_-} \biggl(\frac{1}{\theta
-\theta
_+}-\frac{1}{\theta-\theta_-} \biggr),
\end{eqnarray}
where $\theta_\pm= \eta\pm\sqrt{\eta^2-2\lambda}$. The condition
$\eta
\in[-\sqrt{2\lambda},0)$ is necessary and sufficient for the expression
at \eqref{ex} to be the Laplace transform of a probability measure on
$\mathbb{R}^+$, and we note that this is the same as the condition
$\lambda\in(0,\lambda^*]$ of Theorem~\ref{thm:II}. Under this condition
$\mu_\lambda$ is a mixture of exponentials (or a gamma distribution if
$\eta=-\sqrt{2\lambda}$). This special case was presented in our
earlier paper~\cite{DP}.

\section{Multi-dimensional RIFPT}\label{sec:md-rifpt}
Given a joint
survival function
$\overline H\dvtx(\mbb R^+)^d\to[0,1]$ and a $d$-dimensional L\'
{e}vy process,
a {$d$-dimensional version of the RIFPT problem} is phrased as the
problem to find a probability measure on $\mbb R^d$ and a collection of
increasing continuous functions $I^1,\ldots, I^d$ such that the
following identity holds:
%
\begin{eqnarray}\qquad
\label{eq:mIFPT} P^{\mu} \bigl(\tau^{Y^1} > t_1,
\ldots, \tau^{Y^d}> t_d \bigr) &=& \overline H(t_1,
\ldots, t_d) \qquad\mbox{for all $t_1,\ldots, t_d
\in\mbb R^+$},
\\
Y^i&:=& X\circ I^i \qquad\mbox{for $i=1,\ldots, d$}.
\end{eqnarray}
In order to present a solution, we will impose some structure on the
joint survival function $\overline H$, assuming that
it is from the class of multivariate \emph{generalised frailty survival
functions} that is defined as follows.

\begin{Def*}
A joint survival function $\overline H\dvtx\mbb R^d_+\to[0,1]$ is
called a
($d$-dimensional)
\emph{generalised frailty distribution} if there exists
a random vector $\Upsilon=(\Upsilon_1,\ldots, \Upsilon_m)$ for some
$m\in\mbb N$ such that we have
\[
\overline H(t_1,\ldots, t_d) = E \Biggl[\prod
_{i=1}^d \overline H_i(t_i|
\Upsilon) \Biggr],\qquad  t_1,\ldots, t_d\in\mbb R^+,
\]
where $\overline H_i(\cdot|u)\dvtx\mbb R^+\to[0,1]$, $i=1,\ldots,
d$, $u\in
\mbb U^m$ denotes a collection of
survival functions, where $\mbb U^m$ denotes the image of the random
vector $\Upsilon$.
\end{Def*}

When we denote by $(T_1,\ldots, T_d)$ a random vector with joint
survival function~$\overline H$,
the condition in the definition can be phrased as the requirement that
there exists
a finite-dimensional random vector $\Upsilon$ such that, conditional on
$\Upsilon$, the random
variables $T_1,\ldots, T_d$ are mutually independent. In the context
of credit risk modelling, for example,
one may interpret
the vector $\Upsilon$ as the common factors driving the solvency of a
collection of $d$ companies
(such as economic environment, as opposed to idiosyncratic factors).

We remark that the terminology ``generalised frailty'' is extracted
from the theory of survival analysis (e.g., Kalbfleisch and
Prentice~\cite{KP}) in which \emph{frailty} refers to a common factor
driving the survival probabilities of the individual entities. One of
the commonly studied models is that of \emph{multiplicative frailty}
where the frailty appears as a multiplicative factor in the individual
hazard functions, in which case the conditional individual survival
functions $\overline H_i(\cdot|u)$ take the form $\overline
H_i(\cdot)^u$ for
$u\in\mbb R^+$.

Assume henceforth that $\overline H$ is a $d$-dimensional generalised
frailty survival function, and denote
the corresponding collection of conditional survival functions by
$\{\overline H_i(\cdot|u), i=1,\ldots, d, u\in\mbb U^m\}$ for some
$m\in
\mbb N$. A solution to the
multi-dimensional IFPT of the survival function $\overline H$ can be
constructed by application of the
construction that was used in Corollary~\ref{thm:IFPT} to the
conditional survival functions $\overline H_i(\cdot|u)$.
To formulate this result, let $\{X^{i|u}, i\in\{1,\ldots, d\}, u\in
\mbb U^m\}$ be a collection of
independent L\'{e}vy processes,
each satisfying Assumption~\ref{A1}, and
denote by
$\{\mu_{i}(\cdot|u), i\in\{1,\ldots, d\}, u\in\mbb U^m\}$ the
probability distributions that have Laplace transforms
$\WH\mu_{i}(\cdot|u)$ given by
\[
\WH\mu_{i}(\theta|u) = \frac{\bar\phi_{i|u}(-\lambda_{i|u})}%
{\bar\phi_{i|u}(-\lambda_{i|u}) + \theta}\cdot\Psi
^+_{i|u}(-\lambda _{i|u}, \ii\theta) \qquad\mbox{for some $
\lambda_{i|u}\in\bigl(0,\lambda^*_{i|u}\bigr]$},
\]
where $\bar\phi_{i|u}$, $\Psi^+_{i|u}$, $\lambda^*_{i|u}$ are the
corresponding
left-inverse of the Laplace exponent, positive Wiener--Hopf factor and
minimum of the Laplace exponent
of $X^{i|u}$, respectively.
Finally, let $\{I_i(\cdot|u), i\in\{1,\ldots, m\}, u\in\mbb U^m\}$
denote the collection of time-changes given by
\[
I_i(t|u) = - \frac{1}{\lambda_{i|u}} \log\overline
H_i(t|u),\qquad t\in\mbb R^+.
\]
The solution of the multi-dimensional IFPT is given as follows.

\begin{Thm} \label{thm:mIFPT}
It holds
\begin{eqnarray*}
P \bigl(\tau_0^{Y^1} > t_1,\ldots,
\tau_0^{Y^d} > t_d \bigr) &=& \overline
H(t_1,\ldots, t_d),\qquad t_1,\ldots,
t_d\in\mbb R^+,\mbox{ with}
\\
Y^{i}(t) &=&Y_0^{i|\Upsilon} + X^{i|\Upsilon}
\bigl(I_{i}(t|\Upsilon ) \bigr),\qquad  i=1,\ldots, d,
\end{eqnarray*}
where, conditional on $\Upsilon= u\in\mbb U^m$, the random variable
$Y_0^{i|u}$
follows the probability distribution $\mu_{i}(\cdot|u)$ and
is independent of
the vector $(X^{1|u},\ldots, X^{d|u})$ of L\'{e}vy processes.
\end{Thm}

\begin{pf}
By the tower property of conditional expectations and the fact that,
conditional on the random variable $\Upsilon$,
the set $\{Y^{i|\Upsilon}, i=1,\ldots, d\}$ forms a collection of
independent random variables,
we have for any vector $(t_1,\ldots, t_d)\in(\mbb R^+)^d$
\begin{eqnarray*}
 P \bigl(\tau_0^{Y^1} > t_1,\ldots,
\tau_0^{Y^d} > t_d \bigr) &=& E \Biggl[\prod
_{i=1}^d P \bigl(\tau_0^{Y^i}
> t_i |\Upsilon \bigr) \Biggr]
\\
& =& E \Biggl[\prod_{i=1}^d
P^{\mu_i(\cdot|\Upsilon)} \bigl(\tau _0^{X^{i|\Upsilon}} > I_i(t_i|
\Upsilon) \bigr) \Biggr] \\
&=& E \Biggl[\prod_{i=1}^d
\overline H_i(t_i|\Upsilon) \Biggr] = \overline
H(t_1, \ldots, t_d),
\end{eqnarray*}
where in the second line we used Corollary~\ref{thm:IFPT}.\
\end{pf}
%

\section{Wiener--Hopf factorization and first-passage times}\label{sec:prel0}
\subsection{Preliminaries}
In this subsection, we set the notation and recall some basic
results concerning the Wiener--Hopf factorization of $X$. We refer to
Sato (\cite{Sato}, Chapter~9), for a self-contained account of classical
Wiener--Hopf factorization theory of L\'{e}vy processes and further
references; see also Kuznetsov~\cite{Kuzn} for a recent derivation
using analytical arguments.

Denote by $\Psi$ the characteristic exponent of $X$, that is, the
unique map $\Psi\dvtx\mbb R\to\mbb C$ that satisfies $E[\exp\{\ii
\theta
X_t\}]=\exp\{t\Psi(\theta)\}$ for any $t\in\mbb R^+$.
According
to the L\'{e}vy--Khintchine formula, the characteristic exponent is
given by
%
\begin{equation}
\label{eq:kappa} \Psi(\theta) = \ii\eta\theta- \frac{\sigma^2}{2}
\theta^2 + \int_{\mathbb R}\bigl[\te{\ii\theta z} - 1 - \ii
\theta z\mathbf 1_{\{
|z|<1\}
}\bigr]\nu(\td z), \qquad\th\in\mbb R,
\end{equation}
where $\eta\in\mbb R$, $\sigma^2\in\mbb R^+$ is the variance of the
continuous martingale part of $X$,
and $\nu$ denotes the L\'{e}vy measure of $X$. Under Assumption~\ref
{A1}, the random variable $X_1$ has negative mean
and the L\'{e}vy measure $\nu$ of $X$ satisfies the condition (e.g.,
Sato~\cite{Sato}, Theorem~25.3)
%
\begin{equation}
\label{eq:levym} \int_{(1, \infty)}\te{\e x}\nu(\td x) <\infty.
\end{equation}
Furthermore,
$\Psi$ can be analytically extended to the interior of the strip
\[
\mathcal S=\bigl\{\th\in\mbb C\dvtx\Im(\theta)\in\Theta^o\cup\{0\}
\bigr\},
\]
where $\Im(\theta)$ denotes the imaginary part of $\theta$ and
where $\Theta^{o}$ is the interior of the set $\Theta=\{\theta\in
\mbb
R\dvtx\psi(\theta)<\infty\}$
which is a nonempty interval given Assumption~\ref{A1}; in the case
$\unl\theta= 0$
the exponent $\Psi$ also extends continuously to $\{\th\in\mbb
C\dvtx\Im
(\theta)=0\}$.
The extension of $\Psi$ to $\mathcal S$ will also be denoted by $\Psi$.
The characteristic exponent $\Psi$ is given in terms of the Laplace exponent
$\psi$ of $X$ by $\psi(\theta) = \Psi(-\ii\theta)$ for $\theta
\in\Theta$.

The probability distributions of the running supremum $X^*(t)$
and infimum $X_*(t)$ of $X$ up to time $t$
are related to the characteristic exponent $\Psi$ by
the Wiener--Hopf factorization
of $X$, which expresses $\Psi$ as the product of the Wiener--Hopf factors
$\Psi^+$ and $\Psi^-$
as follows:
%
\begin{equation}
\label{eq:WH} \frac{q}{q-\Psi(\theta)} = \Psi^+(q,\theta)
\Psi ^-(q,\theta),\qquad
\th\in\mbb R, q>0,
\end{equation}
with $\Psi^+(q,\theta)$ given in~\eqref{eq:Psi+def} and the function
$\Psi^-\dvtx(0,\infty)\times\mbb D^-\to\mbb C$ with $\mbb D^-=\{
u\in
\mbb
C\dvtx\Im(u)\leq0\}$,
given by $\Psi^-(q,\theta) = E[\exp\{\ii\theta X_*(e(q))\}]$ for
$\theta
\in\mbb D^-$,
where, as before, $e(q)$ denotes an independent exponential
random variable with mean $q^{-1}$ that is independent of $X$ (e.g.,
Sato~\cite{Sato}, Theorems 45.2, 45.7, Remark~45.9). The factorization
\eqref{eq:WH}
is a direct consequence of the \emph{probabilistic form
of the Wiener--Hopf factorization} of $X$, according to which (i)
$X^*(e(q))$ and $(X - X^*)(e(q))$
are independent and (ii) $(X - X^*)(e(q))$ and $X_*(e(q))$ have the
same distribution.

Since, as noted before, $X$ has negative mean under Assumption~\ref{A1},
$\lim_{t\to\infty} X^*_t$ is almost surely finite,
and Lebesgue's dominated convergence theorem implies that
$E[\exp(\ii\theta X^*_{e(q)})]\to E[\exp(\ii\theta X^*_{\infty})]$
so that $\Psi^+(0,\theta):=\lim_{q\searrow0}\Psi^+(q,\theta)$ is
well defined
and equal to $E[\exp(\ii\theta X^*_{\infty})]$.
It follows thus from the Wiener--Hopf factorization~\eqref{eq:WH} that
the limit
$\Psi^-(q,\theta)/q$ for $q\searrow0$ exists and is equal to
%
\begin{equation}\qquad
\label{eq:Psimd} \Psi^-(0,\theta)/0:=\lim_{q\downarrow0}q^{-1}
\Psi^-(q,\theta) = -\Psi(\theta)^{-1}\cdot\Psi^+(0,\theta)^{-1},\qquad
\theta\in\mbb R.
\end{equation}
The function $\Psi^+(q,\cdot)$ with $q\in\mbb R^+$ admits an analytical
extension to
the domain $\mathcal S^+:=\{\theta\in\mbb C\dvtx\Im(\theta)>
-\overline \theta\}$,
while the
function $\Psi^-(q,\cdot)/q$ with $q\in\mbb R^+$, may be extended
analytically to
$\mathcal S^-:=\{\theta\in\mbb C\dvtx\Im(\theta)\in(-\infty,
-\unl\theta
)\}$.
Denoting these analytical extensions also by $\Psi^+(q, \cdot)$ and
$\Psi^-(q, \cdot)/q$
the Wiener--Hopf factorization~\eqref{eq:WH} continues to hold for all
$\theta$ in the strip $\mathcal S$.

\subsection{Wiener--Hopf factorization under the Esscher-transform}
In order to establish that $\Psi^+(q,s)$ admits an analytical extension
in $q$
as stated in the \hyperref[sec:intro]{Introduction}, we first provide a
``change-of-variable'' formula
relating $\Psi^+$ to its counterparts under Esscher-transforms of $P$.
We recall that the \emph{Esscher transform} $P_x^{(\theta)}$ of the
probability measure $P_x$ for $x\in\mbb R^+$ and $\theta\in\Theta
:=\{
\theta\in\mbb R\dvtx\psi(\theta)<\infty\}$
is the probability measure that is absolutely continuous with respect
to $P_x$ with Radon--Nikodym derivative
on $\mathcal F_t$ given by
\[
\frac{\td P_x^{(\theta)}}{\td P_x}\bigg\vert _{\mathcal F_t} = \exp\bigl(
\th(X_t-x) - t \psi(\theta)\bigr),\qquad \theta\in\Theta, x\in\mbb R^+.
\]
Under the measure $P^{(\theta)}_x$, the process $X-X_0$
is still a L\'{e}vy process with a Laplace exponent
$\psi^{(\theta)}$ that is given in terms of $\psi$ by
%
\begin{equation}
\label{eq:psistar} \psi^{(\theta)}(s) = \psi(s+\theta) - \psi(\theta),\qquad s+\theta
\in\Theta,
\end{equation}
and with a positive Wiener--Hopf factor denoted by $\Psi^+_\theta$.

\begin{Lemma}\label{lem:chmeasure}
For any $q\in\mbb R^+$ and $\theta\in\Theta$ with $\psi(\theta)<
q$, we have
%
\begin{eqnarray}
\label{eq:chvar} \Psi^+(q, s) &=& \frac{\Psi^+_\theta(q-\psi(\theta),s + \ii\theta
)}{\Psi
^+_\theta(q-\psi(\theta), \ii\theta)},
\nonumber
\\[-8pt]
\\[-8pt]
\nonumber
  \Psi^-(q, s) &=&
\frac{\Psi^-_\theta(q-\psi(\theta),s + \ii\theta
)}{\Psi
^-_\theta(q-\psi(\theta), \ii\theta)},
\end{eqnarray}
for $s\in\mathcal S^+$ and $s\in\mathcal S^-$, respectively.
In particular, we have for any $q\in\mbb R^+$ and $\lambda\in
(0,\lambda^*]$
%
\begin{equation}
\label{eq:chvar2} \Psi^\pm(q, s) = \frac{\Psi^\pm_r(q + \lambda,s + \ii r)}{\Psi
^\pm
_r(q+\lambda, \ii r)},\qquad  r=\bar\phi(-
\lambda),
\end{equation}
for $s\in\mathcal S^+$ and $s\in\mathcal S^-$, respectively.
\end{Lemma}

\begin{pf}
By changing measure from $P$ to $P^{(\theta)}$, we find with $\zeta
=q-\psi(\theta)$
\begin{eqnarray*}
\Psi^+(q, s) &=& \int_0^{\infty}q\te{-qt} E\bigl[
\te{\ii s X^*_t}\bigr]\,\td t = \frac{q}{\zeta} \int
_0^{\infty} \zeta\te{-\zeta t} E^{(\theta)}\bigl[
\te{-\theta X_t}\te{\ii s X^*_t}\bigr]\,\td t
\\
&=& \frac{q}{\zeta} E^{(\theta)}\bigl[\te{-\theta\bigl(X_{e(\zeta)}-X^*_{e(\zeta)}
\bigr)}\te{\ii (s+\ii \theta) X^*_{e(\zeta)}}\bigr]
\\
&=& \frac{q}{\zeta} E^{(\theta)}\bigl[\te{-\theta\bigl(X_{e(\zeta
)}-X^*_{e(\zeta
)}
\bigr)}\bigr] E^{(\theta)}\bigl[\te{\ii(s+\ii\theta) X^*_{e(\zeta)}}\bigr]
\\
&=& \frac
{q}{\zeta} \Psi^-_\theta(\zeta,\ii\theta)
\Psi^+_\theta(\zeta ,s+\ii\theta ) = \Psi_\theta^+(\zeta,\ii
\theta)^{-1}\Psi^+_\theta(\zeta,s+\ii \theta),
\end{eqnarray*}
where we used the probabilistic form of the Wiener--Hopf factorization
of $X$
and the form \eqref{eq:psistar} of $\psi_\theta$
in the third and fourth lines.
The identity concerning $\Psi^-$ is derived in an analogous manner.
Finally, equality~\eqref{eq:chvar2} follows by taking $\theta=r$
in \eqref{eq:chvar}.
\end{pf}

\begin{Lemma}\label{lem:anex}
The functions
$\Psi^+(u, v)$ and $\Psi^-(u, w)$
can be uniquely extended by analytical continuation and continuous extension
to the respective domains
\begin{eqnarray*}
 \mbb V_+&:= &\bigl\{(u,v)\in\mbb C^2\dvtx\Re(u) \ge-\lambda^*,
\Im(v) \geq -\theta^*\bigr\}\setminus\bigl\{\bigl(-\lambda^*,-\ii\theta^*\bigr)
\bigr\},
\\
 \mbb V_-&:=& \bigl\{(u,w)\in\mbb C^2\dvtx\Re(u) \ge-\lambda^*,
\Im(w) \leq 0\bigr\}.
\end{eqnarray*}
In particular, denoting these extensions again by $\Psi^+$ and $\Psi^-$
we have continuity in $\lambda$ of
$\Psi^+(-\lambda,\ii u)$ on $(0,\lambda^*]$
for each $u\in\mbb R^+$ and it holds
%
\begin{eqnarray}
\label{WHpid} \Psi^+(-\lambda,\ii u) &=& \frac{\Psi_r^+(0,\ii(u+r))}{\Psi
_r^+(0,\ii
r)}, \qquad r=\bar\phi(-
\lambda),\lambda\in\bigl(0,\lambda^*\bigr),
\\
\label{eq:WH2} \frac{\lambda}{\lambda+ \Psi(u)} &=&
\Psi^+(-\lambda,u)\Psi ^-(-\lambda,u),\qquad
\lambda\in(0,\lambda^*].
\end{eqnarray}
\end{Lemma}

\begin{pf}
The Wiener--Hopf factor $\Psi^+(q,s)$ is well known to be holomorphic
and nonzero on the domain
$D:= \{(q,s)\in\mbb C^2\dvtx\Re(q)>0, \Im(s)> 0\}$ and continuous
on the
closure $\overline D$.
The identity in \eqref{eq:chvar} implies that at any $(q,s)\in
\overline D$
the power series in
$(q,s)$ of $\Psi^+(q,s)$ and $L(q,s):=\Psi^+_{\theta^*}(q-\psi
(\theta
^*),s + \ii\theta^*)/\Psi^+_{\theta^*}(q-\psi(\theta^*), \ii
\theta^*)$
are equal. Since $L$ is holomorphic on the interior of $\mbb V^+$ and
continuous on $\mbb V^+$,
it follows that the function $\Psi^+(q,s)$ can be uniquely extended by
analytical continuation and continuous extension
to the set $\mbb V_+$. In particular, it follows that the function
$\lambda\mapsto\Psi^+(-\lambda,\ii\theta)$ is continuous on
$(0,\lambda
^*]$, and we have consistency with \eqref{WHpid} by construction of the
extension.
The proof of the extension of $\Psi^-$ to $\mbb V^-$ is similar and omitted.
By multiplying the functions in \eqref{eq:chvar2} with $q=-\lambda$
and using the form of $\Psi^-_r(0,\theta)/0$ [see \eqref{eq:Psimd}],
it follows
that the product in the right-hand side of \eqref{eq:WH2} is equal to
$\{-\Psi_r(u + \ii r) \}^{-1} \{-\Psi_r(\ii r)\} = \lambda/[\Psi
(u)+\lambda]$
[in view of \eqref{eq:psistar}].
\end{pf}
%
\section{Proof of Theorem \texorpdfstring{\protect\ref{thm:II}}{2.6}}\label{sec:pfthmii}
We show first that $\WH\mu_\lambda$ is the Laplace transform of a
probability measure $\mu_\lambda$
and identify this measure in terms of the invariant distribution of the
reflected process
$Z:= X - X_*\wedge0$ (with $x\wedge0=\min\{x,0\}$ for $x\in\mbb R$)
under a certain Esscher transform.

We recall that, since $Z_t$ and $X^*_t$ have the same distribution
under $P_0$ for each fixed $t\ge0$
(by the time-reversal property of L\'{e}vy processes, e.g.,
Bertoin~\cite{Bertoin}, Proposition VI.3)
and, under Assumption~\ref{A1}, $X^*_t$ converges to an almost surely
finite limit $X^*_\infty$ as $t\to\infty$,
the limit $P(Z_\infty\in\td x):=\lim_{t\to\infty}P_0(Z_t\in\td x)$
is well defined and has characteristic function $\Psi^+(0,\theta) =
E[\exp(\ii\theta X^*_\infty)]$.
It is straightforward to verify that the measure $P(Z_\infty\in\td x)$
is the unique invariant probability
distribution of the reflected process $Z$.

For any $\lambda\in(0,\lambda^*)$, we specify the measure
$\mu_\lambda$ on $(\mbb R^+,\mathcal B(\mbb R^+))$ by
%
\begin{eqnarray}
\label{eq:mulambda}  \mu_\lambda(\td x) &= &c_r \cdot r\exp(-r
x)P^{(r)}(Z_\infty\leq x)\,\td x, \qquad x\in\mbb R^+, \mbox{ with}
\nonumber
\\[-8pt]
\\[-8pt]
\nonumber
 c_r&=&1/E^{(r)}\bigl[\exp(-r Z_\infty)\bigr],\qquad
r = \bar\phi(-\lambda ),
\nonumber
\end{eqnarray}
where $\bar\phi$ denotes the inverse of the Laplace exponent as
described above,
and where we used that the mean $E^{(r)}[X_1]$ of $X$ is strictly
negative under
$P^{(r)}$. Here, the normalising constant $c_r$ is
such that any of the measures $\mu_\lambda$ has unit mass. We also define
a measure $\mu_{\lambda^*}$ as the limit in distribution of $\mu
_\lambda
$ for
$\lambda\nearrow\lambda^*$ [the existence of this limit is verified in
Lemma~\ref{lem:LT}(ii)].
We next verify that the function $\WH\mu_\lambda$ defined in \eqref{eq:mul}
is equal to the Laplace transform of the measure
$\mu_\lambda$.

\begin{Lemma}\label{lem:LT}
\textup{{(i)}} For any $\lambda\in(0,\lambda^*)$,
the Laplace transform of $\mu_\lambda$
is given by
%
\begin{eqnarray}\qquad
\label{eq:mul2}  \int_0^\infty\te{-\theta x}
\mu_\lambda(\td x) &=& \frac{\bar
\phi
(-\lambda)}%
{\bar\phi(-\lambda) + \theta}\cdot\frac{\Psi_r^+(0,\ii(\theta
+r))}{\Psi
_r^+(0,\ii r)}, \qquad r=\bar
\phi(-\lambda)\quad \mbox{and}
\\
 \Psi^+(-\lambda,\ii\theta)& =& \frac{\Psi_r^+(0,\ii(\theta
+r))}{\Psi
_r^+(0,\ii r)}, \label{eq:Psipl}
\end{eqnarray}
where $\bar\phi$ denotes the inverse of the Laplace exponent as
described above.

\textup{{(ii)}} $\WH\mu_{\lambda^*}:= \lim_{\lambda\nearrow\lambda
^*}\WH\mu
_\lambda$ is the Laplace transform of
a probability measure and
%
\begin{equation}
\label{mulambdastar} \WH\mu_{\lambda^*} = \frac{\bar\phi(-\lambda^*)}%
{\bar\phi(-\lambda^*) + \theta}\cdot\Psi^+\bigl(-
\lambda^*,\ii\theta\bigr).
\end{equation}
\end{Lemma}
\begin{pf}
(i) It is straightforward to verify from \eqref{eq:mulambda} that $\mu
_\lambda$
is equal to a convolution
%
\begin{eqnarray}
\label{eq:mu22} \mu_\lambda(\td x) = c_r \int
_{[0,x]} r\exp\bigl(-r (x-y)\bigr)E^{(r)} \bigl[\exp(-r
Z_\infty) I_{\{Z_\infty\in\td y\}} \bigr]\,\td x,
\nonumber
\\[-8pt]
\\[-8pt]
\eqntext{x\in\mbb R^+, r = \bar\phi(-
\lambda),}
\end{eqnarray}
so that we obtain the expression \eqref{eq:mul} by taking Laplace
transform in $x$ in \eqref{eq:mu22}.
Equation \eqref{eq:Psipl} directly follows from Lemma~\ref{lem:anex}.

(ii)
As $\bar\phi(-\lambda)$ and $\lambda\mapsto\Psi^+(-\lambda,\ii
\theta)$
are continuous on $(0,\lambda^*]$,
we have thus from \eqref{eq:mul} and \eqref{WHpid} that $\WH\mu
_\lambda
(\theta)$
converges to the expression on the right-hand side of \eqref
{mulambdastar} as $\lambda\nearrow\lambda^*$, for any $\theta\in
\mbb R^+$.
Since $\WH\mu_{\lambda^*}(0)\to1$ when $\theta\searrow0$, the
continuity theorem ({e.g.},
Feller~\cite{Feller}, Theorem XIII.1.2) implies that $\WH\mu
_{\lambda
^*}$ is
the Laplace transform of a probability measure, $\mu_{\lambda^*}$ say,
and $\mu_\lambda$ converges weakly to~$\mu_{\lambda^*}$.
\end{pf}

We next establish for any $\lambda\in(0,\lambda^*]$
the $\lambda$-invariance of the measure $\mu_\lambda$ for the killed
process $\{X_t, t<\tau^X_0\}$
by deriving the joint asymptotic distribution of $Z_t$ and $X_*(t)$,
under the initial distribution $\mu_\lambda$, as $t$ tends to infinity,
conditional on $X_*(t)$ being positive, and identifying the asymptotic
marginal distribution corresponding to $X_*(t)$ as exponential.

\begin{Prop}\label{prop:linv} Let $\lambda\in(0,\lambda^*]$, $r =
\bar
\phi(-\lambda)$ and $\theta,\eta\in\mbb R^+$.

\textup{{(i)}} As $t\to\infty$, $P^{\mu_\lambda}(X_*(t) \ge y| X_*(t)
\ge
0)\to
\te{-r y}$ for $y\in\mbb R^+$, and we have
%
\begin{eqnarray}
\label{cond.conv} E^{\mu_\lambda}\bigl[\te{-\theta Z_t - \eta X_*(t)}|
X_*(t) \ge0\bigr] &\longrightarrow& \Psi^+(-\lambda,\theta) \cdot
\frac{\bar\phi
(-\lambda
)}{\bar\phi(-\lambda) + \eta+ \theta},
\\
\te{\lambda t} P^{\mu_\lambda}\bigl(X_*(t) \ge0\bigr) &\longrightarrow& 1.
\label{cond.conv2}
\end{eqnarray}

\textup{{(ii)}} The probability measure $\mu_\lambda$ is a $\lambda$-invariant
distribution for $\{X_t, t<\tau^X_0\}$.
\end{Prop}

\begin{pf}
(i) We consider first the case $\lambda\in(0,\lambda^*)$.
We find by inserting the definition \eqref{eq:mulambda} of $\mu
_\lambda$,
changing measure from $P$ to the Esscher transform $P^{(r)}$
and interchanging the order of integration (justified by Fubini's theorem)
%
\begin{eqnarray}\label{conv.limit}
&&E^{\mu_\lambda}\bigl[\te{-\theta Z_t-\eta
X_*(t)}\mathbf 1_{\{
X_*(t)\ge0\}}\bigr]\nonumber
\\
\nonumber
&&\qquad= \int_{\mbb R^+} \int_{[0,x]} r
\te{-rx} c_r P^{(r)}(Z_\infty\in\td y)
E_0\bigl[\te{-\theta(Z_t+x) -\eta\bigl(X_*(t)+x\bigr)}
\mathbf 1_{\{X_*(t)\ge-x\}
}\bigr]\,\td x
\\
&&\qquad= c_r \int_{\mbb R^+} \int
_{[0,x]} r\te{-(r+\eta+ \theta) x} P^{(r)}(Z_\infty
\in\td y) \te{-\lambda t} \nonumber\\
&&\hspace*{83pt}{}\times E^{(r)}_0\bigl[\te{-(\theta+r)
Z_t - (\eta +r)X_*(t)}\mathbf 1_{\{X_*(t)\ge-x\}}\bigr] \,\td x
\nonumber\\
&&\qquad= \te{-\lambda t}\\
&&\qquad\quad{}\times c_r \int_{\mbb R^+}
\int_{y}^\infty r\te {-(r+\eta+\theta) x}\nonumber\\
&&\hspace*{93pt}{}\times E^{(r)}_0\bigl[\te{-(\theta+ r) Z_t - (
\eta+r)X_*(t)}\mathbf 1_{\{
X_*(t)\ge
-x\}}\bigr]\,\td x\, P^{(r)}(Z_\infty
\in\td y)
\nonumber\\
 &&\qquad= \te{-\lambda t}\cdot\frac{r}{r+\eta+ \theta}\nonumber\\
 &&\qquad\quad{}\times \int
_{\mbb
R^+} c_r E^{(r)}_0\bigl[
\te{-(\theta+r)Z_t - (\eta+r)\bigl\{\bigl(y+X_*(t)\bigr)\vee0\bigr\}}
\bigr] P^{(r)}(Z_\infty\in\td y),\nonumber
\end{eqnarray}
for $\theta\in\mbb R^+$, with $x\vee0=\max\{x,0\}$ for $x\in\mbb R^+$
and, as before,
$c_r= 1/E^{(r)}\times\break [\exp(-r Z_\infty)]$ and $r = \bar\phi(-\lambda)$.
Since the integrand in
\eqref{conv.limit} tends to $c_rE_0^{(r)}\times\break [\te{-(r+\theta)Z_\infty}]$
when $t$ tends to infinity [which is equal to $\Psi^+(-\lambda,\ii
\theta
)$ by~\eqref{WHpid}],
we deduce by an application of Lebesgue's dominated convergence theorem
that also the integral tends to this constant.
Taking $\eta= \theta=0$ in \eqref{conv.limit}
yields~\eqref{cond.conv2}, and subsequently dividing \eqref
{conv.limit} by
$P^{\mu_\lambda}(X_*(t) \ge0)$ yields~\eqref{cond.conv}.
Finally, we note that the first assertion in (i) is a direct
consequence of the continuity theorem
(e.g., Feller~\cite{Feller}, Theorem XIII.1.2)
and \eqref{cond.conv} (with $\theta=0$).

The case $\lambda=\lambda^*$ can be treated by following the line of
reasoning in the previous paragraph,
replacing throughout the measure
$\mathbf 1_{\mbb R^+}(y) c_r \te{-r y}\times\break P^r(Z_\infty\in\td y)$ by the one
with Laplace
transform $\Psi^+(-\lambda^*,\ii\theta)$.

(ii) The continuity theorem and \eqref{cond.conv} (with $\eta=0$)
implies that we have
$E^{\mu_\lambda}[f(X_t)| X_*(t) \ge0] \to\int_{\mbb R^+} f(x) \mu
_\lambda(\td x)$ as $t\to\infty$
for any continuous\break bounded function $f$ on $\mbb R^+$.
The Skorokhod embedding theorem implies that this convergence remains
valid for
any function $f$ that is bounded and continuous on
$\mbb R^+\setminus C$ with $C$ a countable set, which satisfies $\mu
_\lambda(C)=0$ by
absolutely continuity of
$\mu_\lambda$. Thus, by the Markov property and \eqref{cond.conv2} we
have for $t, \theta\in\mbb R^+$
%
\begin{eqnarray}
\label{pfexist}&& E^{\mu_\lambda}\bigl[\te{-\theta X_t}\mathbf
1_{\{X_*(t)\ge0\}}\bigr]\nonumber\\
&&\qquad= \lim_{s\to\infty} \te{\lambda s} \int
_{\mbb R^+} E_x\bigl[\te {-\theta X_t}\mathbf
1_{\{X_*(t)\ge0\}}\bigr] P^{\mu_\lambda}\bigl(X_s\in\td x, X_*(s)\ge0
\bigr)
\nonumber
\\[-8pt]
\\[-8pt]
\nonumber
&&\qquad= \lim_{s\to\infty} \te{\lambda s} E^{\mu_\lambda}\bigl[
\te{-\theta X_{t+s}}\mathbf 1_{\{X_*(t+s)\ge0\}}\bigr]
\\
&&\qquad= \te{-\lambda t} \lim_{s\to\infty} E^{\mu_\lambda}\bigl[\te {-
\theta X_{t+s}}|X_*(t+s)\ge0\bigr] = \te{-\lambda t} \cdot \WH
\mu_\lambda(\theta).
\nonumber
\end{eqnarray}
Inverting Laplace transforms on the left-hand side and right-hand side
of \eqref{pfexist}
shows that the measure $\mu_\lambda$ satisfies \eqref{eq:linv}
in Definition~\ref{def:linv}, and the proof is complete.
\end{pf}

With the above results in hand, we now move to the question of
uniqueness of the quasi-invariant distributions.

\begin{Prop} \label{lem:unique}
For any $\lambda$ in the interval $(0,\lambda^*]$, there exists a
unique probability measure on $(\mbb R^+, \mathcal B(\mbb R^+))$ that
satisfies the relation
%
\begin{equation}
\label{eq:qinvrel} \mu(A) = \frac{q+\lambda}{q} P^\mu\bigl[
X_{e(q)}\in A, e(q) < \tau^X_0 \bigr],\qquad A\in
\mathcal B\bigl(\mbb R^+\bigr), q>0.
\end{equation}
\end{Prop}

The proof rests on a contraction argument.

\begin{pf*}{Proof of Proposition~\ref{lem:unique}}
Again we consider first the case $\lambda\in(0,\lambda^*)$. By changing
measure from $P$ to the
Esscher transform $P^{(\theta^*)}$, the right-hand side of \eqref
{eq:qinvrel} can be expressed as
\[
\int_{\mbb R^+} \int_{\mbb R^+} (q+\lambda)
\te{-qt} \te{-\lambda ^* t} E^{(\theta^*)}_x\bigl[\te{-
\theta^*(X_t-x)}\mathbf 1_{\{t < \tau^X_0\}}\bigr] \,\td t\, \mu(\td x).
\]
Denote by $\mathcal M$ the collection of measures $m$
on the measure space $(\mbb R^+, \mathcal B(\mbb R^+))$ that satisfy
the conditions
%
\begin{eqnarray}
\label{nu-int} & \qquad\mbox{the measure $\tilde m$ given by $\tilde m(\td x):= \te {-
\theta ^*x}m(\td x)$ satisfies $\tilde m\bigl(\mbb R^+\bigr) = 1$, }&
\\
& \displaystyle\mbox{and }\int_{\mbb R^+}P_x\bigl(e(q) <
\tau^X_0\bigr)\tilde m(\td x) = \frac
{q}{q+\lambda}.&
\label{nu-exp}
\end{eqnarray}
The set $\mathcal M$ is nonempty as it contains the measure $m_\lambda:=
\te
{\theta^* x}\mu_\lambda(\td x)$
(which is the case since $\mu_\lambda$ is a $\lambda$-invariant
distribution of $\{X_t, t<\tau^X_0\}$ by Proposition~\ref{prop:linv}).
Furthermore, $\mathcal M$ is a closed subset of the space $\mathcal
P^*$ of
measures $m$ on $\mbb R^+$
satisfying the integrability condition $\int_{\mbb R^+}\te{-\theta
^*x}m(\td x)<\infty$,
which is a Banach space under the norm given by
$\|\pi- \pi'\|:= \sup_{\Upsilon} |\pi(f)- \pi'(f)|$ with
$\Upsilon:= \{f\in L^0(\mbb R^+)\dvtx|f(x)|\leq\te{-\theta^* x}
\ \forall
x\in\mbb R^+\}$
which is contained in the set $L^0(\mbb R^+)$ of real-valued
Borel-functions with domain $\mbb R^+$.

Next, we let $\mathcal H$ be the operator
$\mathcal H\dvtx\mathcal M \to\mathcal P^*$ given by
%
\begin{eqnarray}
\label{H} (\mathcal H m) (A) = \frac{q+\lambda}{q^*} \int_{\mbb R^+}
\int_{\mbb R^+} q_*\te{-q_* t} P_x^{(\theta^*)}
\bigl(X_t\in A, t<\tau^X_0\bigr)\,\td t\, m(\td
x),
\nonumber
\\[-8pt]
\\[-8pt]
\eqntext{ A\in\mathcal B\bigl(\mbb R^+\bigr), m\in\mathcal M,}
\end{eqnarray}
where $q^*=q+\lambda^*$. We note that any $\lambda$-invariant
distribution $\mu$ of
$\{X_t, t<\tau^X_0\}$ gives rise to a fixed point of $\mathcal H$ in
$\mathcal M$:
denoting by $m_*$ the Borel measure on $\mbb R^+$ given by
$m_*(\td x) = \te{\theta^* x}\mu(\td x)$, it is straightforward to
verify by a change-of-measure argument that
the equality in \eqref{eq:qinvrel} can be equivalently rephrased as
$m_* = \mathcal H m_*$.
We show next that the operator $\mathcal H$ is a contraction on
$\mathcal M$.

First, we verify that $\mathcal H$ maps $\mathcal M$ to itself.
Indeed, for any
$m\in\mathcal M$,
the measure $m'$ on $\mbb R^+$ given by $m'(\td x)=\te{-\theta^*
x}(\mathcal
Hm)(\td x)$
(a) has unit mass and (b) satisfies the condition in \eqref{nu-exp}.
To see that (a) holds we observe that, by changing the measure back
from $P^{(\theta^*)}$ to $P$,
we get
\[
m'(A) = \frac{q+\lambda}{q}P^{\tilde m}\bigl(X_{e(q)}
\in A, e(q)<\tau ^X_0\bigr) = P^{\tilde m}
\bigl(X_{e(q)}\in A| e(q)<\tau^X_0\bigr),
\]
with the measure $\tilde m$ defined in \eqref{nu-int}, where the second
equality follows from~\eqref{nu-exp}.
Furthermore, an application of the Markov property shows
\begin{eqnarray*}
P^{m'}\bigl(\tau^X_0> e(q)\bigr) &=&
E^{\tilde m}\bigl[ P_{X_{e(q)}}\bigl(\tau^X_0>e(q)
\bigr) |\tau^X_0> e(q)\bigr]
\\
&=& P^{\tilde m}\bigl(\tau^X_0 > e(q) +
e'(q)|\tau^X_0> e(q)\bigr) \\
&=&
P^{\tilde
m}\bigl(\tau^X_0> e'(q)
\bigr) = \frac{q}{q+\lambda},
\end{eqnarray*}
where $e'(q)$ and $e(q)$ denote independent $\mathrm {Exp}(q)$-random times
that are independent of $X$,
and where the second line holds as $\tau^X_0\sim\mathrm
{Exp}(\lambda)$
under $P^{\tilde m}$
[since $\tilde m$ satisfies \eqref{nu-exp}]. Hence, also property (b)
holds true.

Second, we note that the definition of $\mathcal H$ yields the estimate
\[
\|\mathcal Hm_1 -\mathcal Hm_2\| \leq\frac{q+\lambda}{q^*}
\|m_1 - m_2\| < \|m_1 - m_2\|,\qquad
m_1, m_2\in\mathcal M,
\]
where in the second inequality we used that $q+\lambda$ is strictly
smaller than
$q^*$.

Thus, an application of Banach's contraction theorem shows
that there exists a unique
measure $\pi^*$ in $\mathcal M$ that satisfies the relation $\pi
^*=\mathcal
H\pi
^*$, which
implies the asserted uniqueness for $\lambda\in(0,\lambda^*)$.

We next consider the boundary case $\lambda=\lambda^*$.
The proof in this case follows by a modification of above argument.
Since the function $v\to\Psi(-\ii v)$ is analytic in a neighbourhood of
$\theta^*$ in the complex plane
and $-\lambda^* = \Psi(-\ii\theta^*)$, and nonconstant analytic
functions map open sets to open sets,
it follows that for any sufficiently small~$\e>0$ and any $\lambda_\e$
satisfying
$\lambda_\e-\lambda^*\in(0,\e]$ there exists an $\upsilon$ in a
neighbourhood of $\theta^*$ in the complex plane
such that $\Psi(-\mathbf i\upsilon)=-\lambda_\e$. Fix such an $\e$
and a
corresponding $\lambda_\e$ and $\upsilon=\upsilon_\e$.
By repeating above argument, replacing the Esscher-transform
$P^{(\theta
^*)}$ by the complex-valued
change of measure $P^{(\upsilon_\e)}$, we find that the corresponding
map $\mathcal H^\e$ [defined by the right-hand side of \eqref{H} with
$(\lambda^*,\theta^*)$ replaced by $(\lambda_\e, \upsilon_\e)$] is
still a contraction
but now on the set $\mathcal M_\e$ of complex valued measures $m=m_1
+ \ii
m_2$ satisfying
the condition $\tilde m(\mbb R^+)=1$ and~\eqref{nu-exp} with the
Borel-measure $\tilde m$ on $\mbb R^+$ now given by
$\tilde m(\td x)= \te{-\upsilon_\e x}m(\td x)$.

Specifically, $\mathcal H^\e$ is a contraction in the Banach space
$\mathcal
P^\e
$ of complex valued measures $m$
satisfying the condition $|\int_{\mbb R^+}\te{-\upsilon_\e x}m(\td
x)|<\infty$, with respect to the norm
$\|\pi- \pi'\|_\e:= \sup_{\Upsilon_\e} |\pi(f)- \pi'(f)|$ where the
supremum is taken
over the subset $\Upsilon_\e:= \{f\in L^0(\mbb C)\dvtx|f(x)|\leq
|\te
{-\upsilon_\e x}|\ \forall x\in\mbb R^+\}$
of the set $L^0(\mbb C)$ of complex-valued Borel functions with domain
$\mbb R^+$.
Thus, also in the case $\lambda=\lambda^*$, Banach's contraction
theorem yields the existence of
a unique probability measure satisfying~\eqref{eq:qinvrel}, and the
proof is complete.
\end{pf*}

\begin{pf*}{Proof of Theorem~\ref{thm:II}}
Let $\lambda$ in $(0,\lambda^*]$ be arbitrary.
In Lemma~\ref{lem:LT} it is shown that $\mu_\lambda$ is the
Laplace transform of the probability measure $\mu_\lambda$.
Furthermore, it follows by combining
Propositions \ref{prop:linv} and \ref{lem:unique}
that the probability measure $\mu_\lambda$
is the unique $\lambda$-invariant distribution for the process $\{X_t,
t<\tau^X_0\}$.
\end{pf*}

\section{Mixed-exponential L\'{e}vy processes}\label{sec:pthecase}
We next identify explicitly the quasi-invariant distributions for the
class of mixed-exponential L\'{e}vy processes that are killed upon
first entrance into the negative half-axis.
We recall that a \emph{mixed-exponential L\'{e}vy process} $X=\{X_t,
t\in
\mbb R^+\}$ is a
jump-diffusion given by
%
\begin{equation}
\label{eq:Mex} X_t = X_0 + \eta t + \sigma
W_t + \sum_{j=1}^{N_t}
U_j, \qquad t\in\mbb R^+,
\end{equation}
where $W$ is a Wiener process, $\eta\in\mbb R$ and $\sigma>0$ denote
the drift and the volatility, and $N$ is a Poisson process with rate
$\ell$ that is independent of $W$.
The series $(U_j)_{j\in\mbb N}$ consists of IID random variables
that are independent of $W$ and $N$ and follow a \emph
{double-mixed-exponential distribution}, which
is a probability distribution on $\mbb R$ with PDF given by
\begin{eqnarray}
f(x) = pf_+(x) + (1-p) f_-(x) \nonumber\\
\eqntext{\displaystyle\mbox{with } f_\pm(x) = \sum
_{k=1}^{m_\pm} a^\pm_k
\alpha^\pm_k\te{-\alpha ^\pm
_k|x|}\mathbf 1_{\mbb R^+}(\pm x),  x\in\mbb R,}
\end{eqnarray}
where $p$ is a number in the unit interval $[0,1]$ and $f_+$ and $f_-$
are themselves probability density functions that are linear
combinations of $m^+$ and $m^-$ exponentials, respectively,
with real-valued weights $a_1^+,\ldots, a_{m^+}$ and $a_1^-,\ldots,
a^-_{m^-}$
and strictly positive parameters $\alpha^+_1,\ldots, \alpha^+_{m^+}$
and $\alpha^-_1,\ldots, \alpha^-_{m^-}$.
To ensure that the functions $f_+$ and $f_-$ are PDFs the parameters $\{
a_k^\pm, k=1,\ldots, m^\pm\}$ need to satisfy certain restrictions;
necessary and sufficient conditions for $f_+$ and $f_-$ to be PDFs are
\[
p_1^\pm> 0,\qquad \sum_{k=1}^{m^\pm}p_k^\pm
\alpha_k^\pm\geq0 \quad\mbox{and}\quad\sum
_{k=1}^{l}p_k^\pm
\alpha_k^\pm\geq0 \qquad\forall l = 1,\ldots,m^\pm,
\]
respectively (see Bartholomew~\cite{Bartholomew1969}).

The characteristic exponent of the L\'{e}vy process $X-X_0$ is given by
\[
\Psi(\theta) = - \frac{\sigma^2}{2}\theta^2 + \ii\eta\theta+ \ell p
\sum_{k=1}^{m^+}a_k^+
\frac{\ii\theta}{\a_k^+ - \ii\theta
} - \ell(1-p) \sum_{j=1}^{m^-}a_j^-
\frac{\theta\ii}{\a_j^- + \theta\ii}.
\]

As the function $\Psi$ is rational, it admits an analytical
continuation to the complement
in the complex plane of the finite set $\{-\ii\alpha^+_1,\ldots,
-\ii
\alpha^+_{m^+}, \ii\alpha^-_1,\break \ldots, \ii\alpha^-_{m^-}\}$,
which is
again denoted by $\Psi$.
The mixed-exponential L\'{e}vy process satisfies Assumption~\ref{A1}
precisely if the parameters satisfy the restriction
%
\begin{equation}
\label{eq:as1r} \psi'(0) = \eta+ \ell p \sum
_{k=1}^{m^+}\frac{a_k^+}{\a_k^+} - \ell (1-p) \sum
_{k=1}^{m^-}\frac{a_k^-}{\a_k^-} < 0.
\end{equation}

The
Wiener--Hopf factors associated to $X$
are given by (from Lewis and Mordecki \cite{LM})
%
\begin{eqnarray}
\label{eq:Psipm} \Psi^+(q,\theta) &=&
\frac{1}{  (1 - {\ii\theta}/{(\rho_0^+(q))}  )} \prod
_{k=1}^{m^+} \frac{  (1 - {\ii\theta}/{\alpha
_k^+}
)}%
{  (1 - {\ii\theta}/{(\rho_k^+(q))}  )},
\nonumber
\\[-8pt]
\\[-8pt]
\nonumber
 \Psi^-(q,\theta)& =&
\frac{1}{  (1 - {\ii\theta}/{(\rho
_0^-(q))}  )} \prod_{k=1}^{m^-}
\frac{  (1 + {\ii\theta}/{\alpha
_k^-}
)}%
{  (1 - {\ii\theta}/{(\rho_k^-(q))}  )},
\end{eqnarray}
for $q>0$, where $\rho_k^+(q), k=0,\ldots, m^+$, and
$\rho_j^-(q), j=0,\ldots, m^-$, are the roots of the
Cram\'{e}r--Lundberg equation
%
\begin{equation}
\label{eq:CL} \Psi(-\ii\theta) - q = 0
\end{equation}
with positive and negative real parts, respectively (where multiple
roots are listed
as many times as their multiplicity). By analytical continuation and
continuous extension
it follows that the expressions in \eqref{eq:Psipm} remains valid for
$q\in[-\lambda^*,0]$. The quasi-invariant distributions are expressed
in terms of
these ingredients as follows.

\begin{Prop}\label{lem:mix}
For any $\lambda\in(0,\lambda^*]$
%
\begin{equation}
\label{eq:mulme} \WH\mu_\lambda(\theta) = \frac{\bar\phi(-\lambda)}{\bar\phi(-\lambda) + \theta}\cdot
\frac{\rho
^+_0(-\lambda)}{\rho^+_0(-\lambda) + \theta} \prod_{j=1}^{m^+}
\frac{  (1 + {\theta}/{\alpha
_j^+}  )}%
{  (1 + {\theta}/{(\rho_j^+(-\lambda))}  )}
\end{equation}
is the Laplace transform
of the $\lambda$-invariant probability distribution $\mu_\lambda$ of
$\{
X_t, t<\tau^X_0\}$,
where $\rho_k^+(-\lambda), k=0,\ldots, m^+$, denote the roots $\rho$
of $\Psi(-\ii\rho)=-\lambda$
with $\Re(\rho)>\bar\phi(-\lambda)$.
\end{Prop}
In the \hyperref[ssec:resd]{Appendix}, we present a self-contained proof of the
$\lambda$-invariance
of $\mu_\lambda$ based on residue calculus.

\begin{Rem} In the case that the roots $\rho_k^+(-\lambda)$ are
all distinct
the probability measure $\mu_\lambda$ is a mixed-exponential distribution
that can be obtained from the Laplace transform $\WH\mu_\lambda$
by partial fraction decomposition and termwise inversion:
%
\begin{eqnarray}
\label{eq:mlme}  \mu_\lambda(\td x) &=& \mathbf 1_{\mbb R^+}(x)\cdot
m_\lambda (x)\,\td x,
\nonumber
\\[-8pt]
\\[-8pt]
\nonumber
 m_\lambda(x) &=& A_0^- \bar
\phi(-\lambda)\te{-\bar\phi(-\lambda) x} + \sum_{k=0}^{m_+}
A^+_k \rho^+_k(-\lambda) \te{-\rho_k^+(-
\lambda ) x}.
\end{eqnarray}
Here, the constants
$A^+_k$, $k=0,\ldots, m_+$, and $A_0^-:= A^+_{-1}$ are given by
%
\begin{equation}
\label{eq:Akme} A^+_k = \biggl(1 - \frac{\rho_k^+(-\lambda)}{\alpha_k^+} \biggr)\cdot
\prod_{j=-1, j\neq k}^{m^+} \frac{  (1 - {\rho
_k^+(-\lambda
)}/{\alpha_j^+}  )}%
{  (1 - {\rho_k^+(-\lambda)}/{(\rho_j^+(-\lambda))}  )},
\end{equation}
where $\rho^+_{-1}(-\lambda):= \bar\phi(-\lambda)$
and the constants $\alpha^+_{-1}$ and $\alpha_0^+$ are to be taken
equal to $+\infty$ [so that the factors
$(1 + \rho_k^+(-\lambda)/\alpha_0^+)$ and $(1 + \rho_k^+(-\lambda
)/\alpha_{-1}^+)$ in the product are equal to 1].
\end{Rem}

\begin{Rem}
The class of mixed-exponential L\'{e}vy processes is dense in the class
of all L\'{e}vy processes
[in the sense of weak convergence in the Skorokhod topology $J_1$ on
the Skorokhod space $D(\mbb R)$], which can be seen as follows. It is
well known
(see, e.g., Jacod and Shiryaev~\cite{JS}, Corollary VII.3.6) that
a sequence of L\'{e}vy processes converges weakly precisely if
the values at time $t=1$ converge in distribution. The corresponding
infinitely divisible distributions may be approximated arbitrarily
closely by a sequence of compound Poisson distributions $\mbox
{CP}(F_n,\ell_n)$
where the distributions $F_n$ may be chosen to be double-mixed
exponential distributions as the latter
form a dense class in the sense of weak convergence in the set of all
probability distributions on the real line (see Botta~and~Harris~\cite{BH}).
\end{Rem}
%
\section{Application to credit-risk modelling}\label{sec:crv}
With the results on inverse first-passage time problems in hand, we
next turn to an application of these results
to the problem of counterparty risk valuation.
As noted in the\break \hyperref[sec:intro]{Introduction}, in the structural
approach that was
initially proposed by Black and Cox~\cite{BC} the time of default of a
firm is defined as the first epoch that the value of the firm falls
below the value of its debt, which in the setting of \cite{BC} is equal
to the hitting time of a geometric Brownian motion to some level.
Subsequent studies
such as \cite{AZ,HW} present stylized ``default barrier models''
for the time of default as the epoch of first-passage of
a stochastic process over a default barrier.

A credit default swap (CDS) is a commonly traded financial contract
that provides insurance against the event that a specific company
defaults on its financial obligations. An important problem for a
financial institution is to ensure that the model-values of traded
credit derivatives (such as the CDS) that are recorded in its books are
consistent with market quotes. In a default-barrier model for the value
of the CDS, one is led to the inverse problem of identifying the
boundary that will equate model- and market-values.

Apart from featuring in the valuation of credit derivatives such
as the CDS, the credit risk of a company may also affect the value of
other assets in the portfolio, especially in the cases where
the company in question acts as counterparty in a trade. The
quantification of this type of risk, named \emph{counterparty risk},
requires the joint modelling of asset values and the risk of
default of the company in question (see Cesari {et al.} \cite
{Cesari} for background on counterparty risk). Various aspects of the
modelling of counterparty risk in default barrier models
have been investigated, for instance, in~\cite{BP,BT,EEH,LS,LS2,OS};
in these papers, the model and market quotes are matched by calibration
of the model parameters. Next, we present an explicit example of the
valuation of a call option under counterparty risk in a default-barrier
model that is \emph{by construction} consistent with a given
risk-neutral probability of default, using the solution to the RIFPT
problem given in Corollary~\ref{thm:IFPT}.

\subsection{Valuation of a call option under counterparty risk}\label
{sec:crv1}
This problem involves three entities, a company $A$, whose stock price
is denoted by $S_t$, a bank $B$ and the bank's counterparty $C$.
The problem under consideration is the fair valuation of the
{counterparty risk} to $B$
resulting from a transaction
in which $C$ has sold to $B$ a European call option on the stock of
company $A$.
We consider the situation where only $C$ is default risky while $A$ and
$B$ are free of default risk---in the finance literature the call
option is in this case referred to as a \emph{vulnerable} call option
(first labelled such by Johnson and Stultz~\cite{JSt}; see also Jarrow
and Turnbull~\cite{JT} for an application to zero-coupon bond valuation).
Then $B$, as the owner of the call option, is exposed to counterparty
risk, namely the potential loss that is incurred if its counterparty
$C$ goes into default before the maturity $T$ of the call option, and
fails to deliver the pay-off of the call option. If $\tau$ denotes the
epoch of default of~$C$, then the {fair value} $\pi$
of the potential loss of the holder of the option (discounted to time 0
at the risk-free rate $r$) and the so-called \emph{expected positive
exposure} $P_t$ are given by
%
\begin{eqnarray}
\label{eq:CPR} \Pi&=& E[V_{\tau}\mathbf 1_{\{\tau\leq T\}}],
\\
P_t &=& E[V_{\tau}|\tau= t],\qquad t\in[0,T],\label{eq:PPd}
\end{eqnarray}
where $V_\tau$ denotes
the value at time $\tau$ of a
$T$-maturity call-option with strike $K$ on the value of stock,
discounted to time 0:
%
\begin{equation}
\label{eq:Vtau} V_\tau=\te{-r\tau}E\bigl[\te{-r(T-\tau)}(S_T-K)^+|
\mathcal F_\t\bigr].
\end{equation}
The conditional expectation in \eqref{eq:PPd}
is understood as the regular version of the conditional expectation
$E[V_{\tau}|\tau]$ [under Assumption~\ref{As2}(iii) below this
conditional expectation can just be defined in the usual way for
continuous random variables].
We will phrase the model in terms of two independent L\'{e}vy processes
$X$ and $Z$ satisfying Assumption~\ref{A1}.
Throughout this section, we work under the following additional assumptions.

\begin{As}\label{As2}
(i) We have $\unl\theta_X<-1$, $\overline \theta_X > 1 +
\alpha$,
$\overline \theta_Z>1+\a$ for some $\alpha>0$.\vspace*{-6pt}
\begin{longlist}[(iii)]
\item [(ii)] The CDF $H$ has a continuous density $h$, and satisfies
$\overline
H(T)>0$ and $\lambda_X^* > -\log\overline H(T)/H(T)$,
where $\lambda_X^*$ denotes the maximizer in \eqref{eq:petrov}.
\item [(iii)] For any $x>0$, there exists a collection of measures
$\{p_{t,x}(\td y), t\in\mbb R^+\}$
on $(\mbb R_-, \mathcal B(\mbb R_-))$ satisfying $p_{t,x}(\td y)\,\td t =
P(X_{\tau^X_{-x}}\in\td y, \tau^X_{-x}\in\td t)$.
\end{longlist}
\end{As}

Let the credit-worthiness of the counterparty $C$ be
described by the \emph{distance-to-default} $Y$, in the sense that the
default of $C$ occurs at the first moment that the process $Y$ falls
below the level $0$.
We assume that the process $Y$ is given in terms of $X$ by
%
\begin{eqnarray}
\label{eq:Yindexex}  Y_t &=& Y_0 + X_{I(t)},\qquad
I(t) = I_{\mu^X_{\lambda^0}}(t) =
T \cdot \frac{\log\overline H(t)}{\log\overline H(T)},\qquad t\in[0,T],
\\
 Y_0 &\sim&\mu^X_{\lambda^0},\qquad
\lambda^0 = - T^{-1}\cdot\log \overline H(T),
\end{eqnarray}
where, as before, $Y_0$ is independent of $X$ and
$\mu_{\lambda^0}^X$ denotes the $\lambda^0$-invariant
distribution of $\{X_t, t<\tau^X_0\}$.
Here, we have chosen $\lambda^0$ so as to normalise the ratio
$I(T)/T$ to unity.
Note that the CDF of the first-passage time $\tau^Y_0$ of the process
$Y$ defined in \eqref{eq:Yindexex} is given by $H$
[in view of Corollary~\ref{thm:IFPT} and Assumption~\ref{As2}(ii)].

In the case that the price process $S$ is independent of credit index
process $Y$, we note that the expectation in \eqref{eq:CPR} is just equal
to the integral of the expectation $E[V_t]$ against the measure
$H(\td t)$. Next, we consider an instance of the complementary case that
$S$ and $Y$ are dependent. More specifically, we assume that $S$ is
given by
%
\begin{equation}\qquad
\label{eq:S} %
\cases{ S_t = S_0 \exp \bigl
\{(r-d) t + L_t - \kappa_{t, I(t)}(-\ii) \bigr\}, &\quad  $t\in
[0,T], S_0>0$,\vspace*{2pt}
\cr
L_t = \rho
X_{I(t)} + Z_t, & \quad $\rho\in[-1,1]$,\vspace*{2pt}
\cr
\kappa_{t_1, t_2}(u) = \Psi_Z(u)t_1 +
\Psi_X(u\rho)t_2, & \quad $\Im(u)\in[-1-\alpha,0]$,} %
\end{equation}
where $\Psi_Z$ and $\Psi_X$ denote the characteristic exponents of the
L\'{e}vy processes $X$ and $Z$ and $r$ and $d$ denote the risk-free rate
and the dividend yield, respectively. The degree of dependence between
the stock price process $S$ and the credit index process $Y$ is controlled
by the parameter $\rho$. Note that $\kappa_t$ has been specified such
that the discounted stock-price process with reinvested dividends $\te
{-r t}[\te{d t} S_t]$ is a martingale. In the following
result, a semi-analytical expression is derived
for $\pi$ and $P(t)$ in terms of an inverse Fourier-transform
$\mathcal
F_\xi^{-1}$
and an inverse Laplace-transform $\mathcal L_q^{-1}$ with respect to
$\xi$
and $q$, respectively.

\begin{Prop}\label{prop:cvacall}
The values $\pi$ and $P_t$, $t\in[0,T]$, are given by
%
\begin{eqnarray}
\label{eq:exppi} \Pi&=& \int_0^{T} N(t)
\frac{h(t)}{\lambda^0\overline H(t)} \,\td t,\qquad  P_t = \frac{N(t)}{\lambda^0 \overline H(t)},
\\
N(t) &=& \te{rT} \mathcal F_\xi^{-1}
\bigl(D_{t,T}(u) C_{t}(u) \bigr) (k),\qquad u = 1 + \alpha+ \ii
\xi,\label{eq:expPt}
\\
C_t(u)& = &\bigl(\lambda^0 \overline H(t)
\bigr)^{-1} \exp\bigl\{(r-t)tu - \kappa _{t,I(t)}(-\ii)u +
\Psi_Z(-\ii u) t\bigr\}
\nonumber
\\[-8pt]
\\[-8pt]
\nonumber
&&{}\times \mathcal L_q^{-1}
\bigl(f_{\rho u}(q)\bigr) (t),
\\
\qquad D_{t,T}(u) &=& \frac{\exp\{\kappa_{T,I(T)}(-\ii u) - \kappa
_{t,I(t)}(-\ii u)\}}{u(u-1)}, \label{eq:Dex}
\end{eqnarray}
with $k = \log K/c'$, $c'=\exp(-rT + (r-d)(T-t) - \kappa_T(-\ii) +
\kappa_t(-\ii))$.
\end{Prop}

The proof relies on the following auxiliary result:

\begin{Lemma}\label{lem:condexm}
For any $u$ with $\Re(u)\in[0,\overline \theta_X)$ and $t\in[0,T]$
we have, with $\tau= \tau_0^Y$,
%
\begin{eqnarray}
\label{eq:EuX}  E \bigl[\te{u X_{I(\tau)}} |\t= t \bigr] &=& \frac{1}{\lambda^0\overline H(t)}
\mathcal L_q^{-1} \bigl(f_{u}(q)\bigr)
\bigl(I(t)\bigr),
\nonumber
\\[-8pt]
\\[-8pt]
\nonumber
f_{u}(q) &=& \int_{\mbb R^+}
\mu^X_{\lambda^0}(\td x) E \bigl[\te{u X_{\tau
^X_{-x}} - q
\tau^X_{-x}} \bigr].
\end{eqnarray}
In particular, for $u$ satisfying
in addition $\Re(u)\in[0,\overline \theta_Z\wedge\overline \theta
_X/\rho)$
we have
%
\begin{eqnarray}
\label{eq:EuS} &&E\bigl[S_\tau^{u}|\tau=t\bigr]\nonumber\\
&&\qquad =
\frac{S_0^{u}}{\lambda^0\overline
H(t)}
\\
&&\qquad\quad{}\times\exp \bigl\{ (r-d)t u - \kappa_t(-\ii)u +
\Psi_Z(-\ii u)t\bigr\}\cdot\bigl(\mathcal L_q^{-1}
f_{\rho u}(q)\bigr) (t).\nonumber
\end{eqnarray}
\end{Lemma}

The function $f_{u}(q)$ can be explicitly expressed in terms of the
Wiener--Hopf factor $\Psi^-$ of $X$ and $\mu_{\lambda^0}^X$ by
deploying the \emph{Pe{\v{c}}erski{\u\i}--Rogozin identity} [see the expression
\eqref{keyid2b} in the \hyperref[ssec:resd]{Appendix}].

\begin{pf*}{Proof of Lemma~\ref{lem:condexm}}
The spatial homogeneity of the L\'{e}vy process $X$ and the definition
of the stopping time
$\tau$ yield $P(X_{I(\tau)}\in\td y, \tau\in\td t) = \int\mu(\td x)
P(X_{\tau^X_{-x}}\in\td y, \tau^X_{-x}\in\td I(t))$. Since the CDF of
$\tau^Y_0$ is given by $H$, it follows thus
by Bayes' lemma that the conditional
expectation in the left-hand side of \eqref{eq:EuX} can be expressed as
%
\begin{equation}
\label{eq:idexc} E\bigl[\te{u X_{I(\tau)}}|\tau= t\bigr] = \frac{1}{h(t)}\int
_{\mbb R^+}\mu ^X_{\lambda^0}(\td x) \int
_{\mbb R} \te{u x} p_{I(t),x}(\td y) I'(t).
\end{equation}
The form of the derivative $I'(t)=h(t)/[\lambda^0\overline H(t)]$ then
implies that the right-hand side of \eqref{eq:idexc} and \eqref{eq:EuX}
are equal. The identity in \eqref{eq:EuS}
follows now as a direct consequence of the form of $S$ in given in
\eqref{eq:S} and the independence of $Z$ and $\tau$.
\end{pf*}

\begin{pf*}{Proof of Proposition~\ref{prop:cvacall}}
Note first that the form of $\pi$ is obtained by integrating $P_t$
against $h(t)$ over the interval $[0,T]$,
performing the change of
variables $u=I(t)$ and using the observation $I'(t)=h(t)/[\lambda
\overline H(t)]$.

The independence of the increments of $\log S$ implies
\begin{eqnarray*}
P_t &= &E\bigl[G_{\tau, S_\tau}(k)|\tau=t\bigr],\qquad
G_{t,s}(k) = s'\te{-rT}\cdot E\bigl[\bigl(
\te{L_T-L_t} - \te{k}\bigr)^+ \bigr],
\\
 s' &=& s c',\qquad  c' = \exp\bigl((r-d)
(T-t) - \kappa_T(-\ii) + \kappa _t(-\ii)\bigr),\qquad  k =
\log\bigl(K/s'\bigr).
\end{eqnarray*}
By a standard Fourier transform argument, it can be shown that
$G_{t,s}(k)$ admits an explicit integral representation
in terms of the characteristic exponents of $X$ and~$Z$.
More specifically, since the dampened function $k\mapsto\exp(\a
k)\cdot G_{t,s}(k)$ and its Fourier transform are integrable,
the Fourier inversion theorem implies
%
\begin{eqnarray}
 G_{t,s}(k) &= &\exp(-\a k) \bigl[\mathcal F_{\xi}^{-1}
\bigl(G^{\wedge}_{t,s}\bigr)\bigr](k),
\nonumber
\\[-8pt]
\\[-8pt]
\nonumber
  G^{\wedge}_{t,s}(
\xi) &=& s' \cdot D_{t,T}(1+\a+ \ii\xi),\qquad \xi\in \mbb R,
\end{eqnarray}
where $D_{t,T}(1+\a+ \ii\xi)$ is given in \eqref{eq:Dex}.
By an interchange of the expectation and integration
(justified by Fubini's theorem), we find that $P_t$, $t\in[0,T]$,
is equal to
%
\begin{eqnarray}
&& P_t = \frac{c'}{2\pi} \int_{\mbb R} E
\bigl[S_\tau^{1+\a+\ii\xi}|\tau=t\bigr]\cdot \biggl(\frac{K}{c'}
\biggr)^{-\alpha-\ii\xi}D_{t,T}(1 + \a+ \ii \xi) \,\td\xi.
\end{eqnarray}
The expression for $P_t$ in \eqref{eq:expPt} then follows by inserting
the expression in \eqref{eq:EuS} in Lemma~\ref{lem:condexm}.
\end{pf*}
%
\subsection{Extensions}
We end this section with a brief description of a number of possible
extensions and related problems in the current model setting. First, we
mention that, in addition to the case of the call option that was
considered above, it is of interest to value the counterparty risk for
other instances of commonly traded securities in foreign exchange,
fixed income, equity or commodity markets, such as swap contracts which
are contracts involving regular payments of both parties that entered
into the contract. Second, we recall that in the setting above it was
assumed that parties $A$ (the company that issued the stock) and $B$
(the bank)
were free of default risk. The case where two of three or all three
parties are subject to default is a natural extension that is
applicable in many situations. Such an extension may still be treated
in the current setting
deploying the solution of the multi-dimensional IFPT in Theorem~\ref{thm:mIFPT}.
Finally, especially of interest to financial market practitioners will
be the development of an
efficient numerical implementation of the model.
In the interest of brevity, these questions are left for future research.

\begin{appendix}\label{ssec:resd}
\section*{Appendix: Proof of quasi-invariance by residue calculus}
In this section, we provide an alternative proof of the $\lambda
$-invariance of the
probability measure $\mu_\lambda$ in the case that $X$ is a
mixed-exponential L\'{e}vy process.
We observe first (in Proposition~\ref{lem2}) that a probability
measure $\mu$ is a $\lambda$-invariant distribution of $\{X_t, t<\tau
^X_0\}$ precisely if its Laplace transform $\WH\mu$ satisfies the identity
%
\setcounter{equation}{0}
\begin{eqnarray}
\label{eq:idd2} \WH\mu(\theta)\cdot\frac{q}{q+\lambda} = \Psi^+(q,\ii\theta) \cdot
\frac{1}{2\pi\ii}\int_{a-\ii\infty}^{a+\ii\infty}\WH\mu(-u)
\Psi^-(q,- \ii u) \frac{\td u}{u + \theta},
\nonumber
\\[-8pt]
\\[-8pt]
\eqntext{ q>0.}
\end{eqnarray}
By an application of Cauchy's residue theorem,
we verify subsequently that the Laplace transform $\WH\mu_\lambda$
given in \eqref{eq:mulme} satisfies the identity in \eqref{eq:idd2} for
any fixed $q>0$.

For the ease of presentation, we restrict to the case that both the
roots $\rho$ of
the equation $\Psi(-\ii\rho)=-\lambda$
and those of the equation $\Psi(-\ii\rho)=q$
are distinct; the case of multiple roots can be dealt with by similar arguments.

\subsection{Integral identity}
We give next an expression in terms of a Bromwich-type integral
for the Laplace transform of $X(e(q))$
on the set $\{X_*(e(q)) \ge0\}$ under a given initial distribution
$\mu$
and use this to derive an integral equation satisfied by
the Laplace transform of a $\lambda$-invariant distribution.
To derive these expressions, we first express
the Laplace transform of the function $K_{\th,q}\dvtx\mbb R^+\to\mbb R$
given by
\[
K_{\th,q}(x) = E_x\bigl[\te{-\theta X\bigl(e(q)\bigr)}
\mathbf 1_{\{\tau^X_0 >
e(q)\}
}\bigr], \qquad x\in\mbb R^+,
\]
for given positive $q$ and $\theta$, in terms of the
Wiener--Hopf factors $\Psi^+$ and~$\Psi^-$.

\begin{Lemm}\label{lem:Khat} \textup{{(i)}} For $\theta, q>0$ and $x\in
\mbb
R^+$ we have
%
\begin{equation}
\label{eq:K11} K_{\theta,q}(x) = \Psi^+(q,\ii\theta) \cdot E_0
\bigl[\te{-\theta X_*\bigl(e(q)\bigr)}\mathbf 1_{\{X_*(e(q)) \ge-x\}}\bigr].
\end{equation}

\textup{{(ii)}} The Laplace transform $\WH K_{\theta,q}$ of $K_{\theta,q}$
is given by
%
\begin{equation}
\label{eq:Khat} \WH K_{\theta,q}(u) = \frac{\Psi^+(q,\ii\theta)\Psi^-(q,-\ii
u)}{\theta+u},\qquad u\in\mbb R^+.
\end{equation}
\end{Lemm}

\begin{pf}
(i) The independence of
the random variables $(X - X_*)(e(q))$ and $X_*(e(q))$ under $P_0$
(from the probabilistic form of the Wiener--Hopf factorization)
and the fact that the events $\{\tau^X_0 > e(q)\}$
and $\{X_*(e(q)) \ge0\}$ are equal $P_x$-a.s. for any nonnegative $x$
[{i.e.}, the probability $P_x(\Delta)$ of the difference $\Delta$
of these two sets is $0$]
imply that we have
\begin{eqnarray*}
\label{eq:K1}%
K_{\th,q}(x) &=& E_x
\bigl[\te{-\theta X\bigl(e(q)\bigr)}\mathbf 1_{\{\tau^X_0 >
e(q)\}}\bigr] = \te{-\theta
x} E_0\bigl[\te{-\theta X\bigl(e(q)\bigr)}\mathbf 1_{\{X_*(e(q)) \ge
-x\}
}
\bigr]
\\
&=& \te{-\theta x} E_0\bigl[\te{-\theta\bigl\{ (X-X_*)
\bigl(e(q)\bigr) + X_*\bigl(e(q)\bigr)\bigr\} }\mathbf 1_{\{X_*(e(q)) \ge-x\}}\bigr]
\\
&=& \te{-\theta x} E_0\bigl[\te{-\theta(X-X_*) \bigl(e(q)
\bigr)}\bigr]E_0\bigl[\te{-\theta X_*\bigl(e(q)\bigr)}\mathbf
1_{\{X_*(e(q)) \ge-x\}}\bigr] 
\end{eqnarray*}
for any nonnegative real $x$, which yields \eqref{eq:K11} in view of
the fact that the Laplace transform
of $(X-X_*)(e(q))$ is given by $\Psi^+(q,\ii\theta)$.

(ii) In view of \eqref{eq:K11}, the Laplace transform $\WH K_{\theta
,q}$ is equal to
\begin{eqnarray*}
\WH K_{\theta,q}(u) &=& \Psi^+(q,\ii\theta) E_0 \biggl[\int
_0^\infty\te{-(u+\theta) x}\te{-\theta X_*
\bigl(e(q)\bigr)}\mathbf 1_{\{
X_*(e(q)) \ge-x\}}\,\td x \biggr]
\\
&=& \Psi^+(q,\ii\theta) E_0 \biggl[\te{-\theta X_*\bigl(e(q)\bigr)}
\int_{-X_*(e(q))}^\infty\te{-(u+\theta) x}\,\td x \biggr]
\\
&=& \Psi^+(q,\ii\theta) \frac{1}{\theta+ u} E_0\bigl[\te{u X_*
\bigl(e(q)\bigr)}\bigr],\qquad  u\in\mbb R^+,
\end{eqnarray*}
which yields \eqref{eq:Khat} by definition of the Wiener--Hopf factor
$\Psi^-$.
\end{pf}

\begin{Propp}\label{lem2}
Let $\mu$ be a probability measure on $\mbb R^+\setminus\{0\}$
without atoms
and denote by~$\WH\mu$ its Laplace transform. Assume that
there are $c>0$, $C>0$ and $a\in\Theta^o$ satisfying $\WH\mu
(-a)<\infty
$ and
%
\begin{equation}
\label{eq:CC} \bigl|\WH\mu(-u) \bigl(1+|u|^c\bigr)\bigr| < C\qquad \mbox{for all
$u$ with $\Re(u)=a$}.
\end{equation}

\textup{{(i)}} For any $q, \theta\in\mbb R^+$, $q\neq0$, we have the identities
%
\begin{eqnarray}
\label{keyid2} &&E^{\mu}\bigl[\te{-\theta X\bigl(e(q)\bigr)}\mathbf
1_{\{X_*(e(q)) \ge0\}}\bigr]
\nonumber
\\[-8pt]
\\[-8pt]
\nonumber
&&\qquad= \Psi^+(q,\ii\theta) \cdot \frac{1}{2\pi\ii}\int
_{a-\ii\infty}^{a+\ii\infty}\WH\mu(-u) \Psi^-(q,- \ii u)
\frac{\td u}{u + \theta},
\end{eqnarray}
\begin{eqnarray}
\label{keyid2b}&&E^{\mu}\bigl[\te{-q\tau_{0}^X +
\theta(X_{\tau_0^X} - X_0)}\bigr]
\nonumber
\\[-8pt]
\\[-8pt]
\nonumber
&&\qquad= \frac{1}{2\pi\ii}\int
_{a-\ii\infty}^{a+\ii\infty}\WH\mu(-u + \theta) \biggl(1 -
\frac{\Psi^-(q,- \ii u)}{\Psi^-(q,- \ii\theta)} \biggr) \frac{\td
u}{u-\theta}.
\end{eqnarray}

\textup{{(ii)}} Let $\lambda\in(0,\lambda^*]$. The measure $\mu$ is a
$\lambda$-invariant
distribution of the process $\{X_t, t<\tau_0^X\}$
if and only if $\WH\mu$ satisfies the
collection of equations
%
\begin{eqnarray}
\label{eq:id} \WH\mu(\theta)\cdot\frac{q}{q+\lambda} = \Psi^+(q,\ii\theta) \cdot
\frac{1}{2\pi\ii}\int_{a-\ii\infty}^{a+\ii\infty}\WH\mu(-u)
\Psi^-(q,- \ii u) \frac{\td u}{u + \theta},
\nonumber
\\[-8pt]
\\[-8pt]
\eqntext{q>0.}
\end{eqnarray}
\end{Propp}

\begin{Remm}
The identity in \eqref{keyid2} is also valid if instead of \eqref{eq:CC}
we require $|q^{-1}\Psi^-(q, -\ii u)|(1+|u|^c)|<C$
uniformly over all $q>0$ and $u$ with \mbox{$\Re(u)=a$}. We note that the boundedness
of $|\Psi^-(q, - \ii u)|(1+|u|)$ over the set of $q>0$ and $u$ with
$\Re(u)=a$
is equivalent to the condition that the L\'{e}vy process $X$ creeps
downwards. This observation follows from
the fact that $X$ creeps downward precisely if the descending ladder
height process has nonzero infinitesimal drift.
\end{Remm}

\begin{pf}
It follows from \eqref{eq:K1}
that the function $x\mapsto\te{\theta x}K_{\th,q}(x)$ is
nondecreasing on $\mbb R^+$
(and has thus at most countably many points of discontinuity).
The Laplace inversion theorem yields that, at any
point of continuity $x$, $K_{\theta,q}(x)$
is equal to the integral of the right-hand side of the identity in \eqref{eq:Khat}
over the Bromwich contour $\Re(u)=a$, {that is},
%
\begin{equation}
\label{eq:KB} K_{\th,q}(x) = \Psi^+(q,\ii\theta) \cdot
\frac{1}{2\pi\ii}\int_{a-\ii\infty}^{a+\ii\infty} \te{ux} \Psi^-(q,-
\ii u) \frac{\td u}{u + \theta}.
\end{equation}
The identity in \eqref{keyid2} follows by integrating \eqref{eq:KB}
against $\mu(\td x)$ and interchanging the order of integration. This
interchange follows by an
application of Fubini's theorem which is justified in view of the
estimate
%
\begin{eqnarray}
\label{eq:Fubini} &&\int_{(0,\infty)} \int_{0-\ii\infty}^{0+\ii\infty}
\biggl\llvert \te{ux} \frac{\Psi^-(q,-\ii u)}{u+\theta}\biggr\rrvert \,\td u\,
\mu(\td x)
\nonumber
\\[-8pt]
\\[-8pt]
\nonumber
&&\qquad\leq \int
_{(0,\infty)}\mu(\td x) \cdot \int_\mbb R C
\frac{\theta+ |u|}{(u^2+\theta^2)(1+|u|^c)}\,\td u < \infty.
\end{eqnarray}
To derive this estimate, we used the bound in \eqref{eq:CC},
that $\mu$ is a probability measure and
the observations (a) $1/(u+ d) = (\bar u + d)/(|\Im(u)|^2+|\Re(u) +
d|^2)$ for any $d\in\mbb R$
and $u\in\mbb C$, with $\bar u$ denoting the complex conjugate of $u$,
and (b) $|\exp\{ux\}|=\exp\{\Re(u)x\}$
for any $x\in\mbb R$ and $u\in\mbb C$. Hence, the proof of
the identity in \eqref{keyid2} is complete.

The identity in \eqref{keyid2b} can be
proved by an analogous line of reasoning (the details of which are omitted)
by deploying the Pe{\v{c}}erski{\u\i}--Rogozin identity
%
\begin{eqnarray}
\label{eq:PR} \int_0^\infty\te{-ux}
E_x\bigl[\te{-q\tau_{0}^X +\theta
X_{\tau
_0^X}}\bigr]\,\td x = \frac{1}{u-\th} \biggl(1 - \frac{\Psi^-(q,-\ii u)}{\Psi^-(q,-\ii
\theta
)}
\biggr),
\nonumber
\\[-8pt]
\\[-8pt]
\eqntext{ u\in\mbb R^+;}
\end{eqnarray}
for a proof, see, {for example}, Sato~\cite{Sato}, Theorem~49.2, or
Alili and Kyprianou~\cite{AA}, Section~3.1, for a probabilistic proof.

(ii) The assertion follows from Definition~\ref{def:linv}
by noting that (a) the left-hand side and right-hand side of
\eqref{eq:id} are equal to the double Laplace transforms
in $(t,x)$ of the measures $m^{(1)}_t$ and $m^{(2)}_t$ on $(\mbb R^+,
\mathcal B(\mbb R^+))$
given by $m^{(1)}_t(\td x) = \exp(-\lambda t)\mu_\lambda(\td x)$
and $m^{(2)}_t(\td x) = P^{\mu_\lambda}(X_t\in\td x, t<\tau^X_0)$,
respectively [by \eqref{keyid2}] and (b) for any Borel set $A$, $m^{(1)}_t(A)$
and $m^{(2)}_t(A)$ are continuous and c\`{a}dl\`{a}g at any $t>0$,
respectively.
\end{pf}

\subsection{Residue calculus}
We next describe the form of the integrand of the Bromwich integral in
\eqref{eq:idd2} in the case of a mixed-exponential L\'{e}vy process and
$\mu=\mu_\lambda$. Since the positive Wiener--Hopf factor and the
function $\WH\mu_\lambda(\theta)$
are both rational [{cf.} \eqref{eq:mul} and \eqref{eq:Psipm}]
also the function $f\dvtx\mbb C^{+}\to\mbb C$ given by
%
\begin{equation}
\label{eq:f} f(u) = f_{\th,\lambda,q}(u) = \frac{\Psi^+(q,\ii\theta)\WH\mu
_\lambda
(-u)\Psi^-(q,-\ii u)}{u+\theta}
\end{equation}
is rational, for any triplet $(\theta,\lambda,q)$ with $\theta\in
(\unl
\th,\overline \theta)$, $\lambda\in(0,\lambda^*]$ and $q>0$.
Moreover, the
collection of poles of $f$ is finite and given
by $\mathcal P^+\cup\mathcal P^-$ with
%
\begin{eqnarray}
\label{lem2a1} \mathcal P^+& =&\bigl\{\bar\phi(-\lambda)\bigr\}\cup\bigl\{
\rho_k^+(-\lambda); k=0,\ldots, m^+\bigr\}\subset \mbb
C^{++},
\nonumber
\\[-8pt]
\\[-8pt]
\nonumber
\mathcal P^-& =& \bigl\{-\theta, \rho_j^-(q), j=0,
\ldots, m^-\bigr\}\subset \mbb C^{--},
\end{eqnarray}
where we denote $\mbb C^{--}:=\{u\in\mbb C\dvtx\Re(u)<a\}$ and
$\mbb C^{++}:=\{u\in\mbb C\dvtx\Re(u)>a\}$ where $a$ is some fixed
arbitrary number in the interval
$(0,\bar\phi(-\lambda))$.

Denote by $\mathcal C^+_T$ the contour with clockwise orientation consisting
of the segment
$\mathcal I_T = \{u\in\mbb C\dvtx\Im(u)\in[-T,T], \Re(u) = a\}$
and the semi-circle
that joins $a-\ii T$ and $a+\ii T$ such that $\mathcal C^+_T$ is contained
in the set $\{u\in\mbb C\dvtx\Re(u)\ge a\}$.
For $T$ sufficiently large, the contour $\mathcal C^+_T$ encloses
all the poles in the set $\mathcal P^+$.
Next, we evaluate the contour integral of $f$
over the curve $\mathcal C^+_T$.

\begin{figure}

\includegraphics{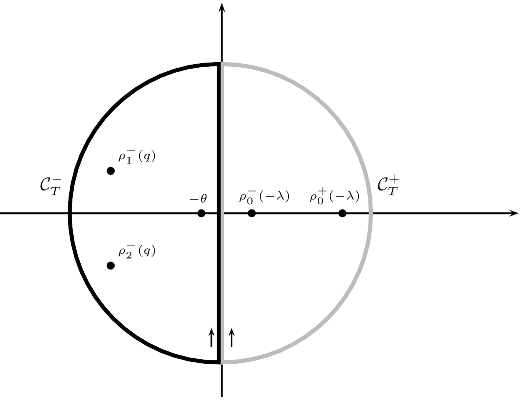}

\caption{Pictured is the complex plane with
an example of the two contours $\mathcal C_T^+$ (grey) and
$\mathcal C_T^-$ (black) and the poles in $\mathcal P^+$ and
$\mathcal P^-$.
The contour $\mathcal C_T^+$ has a clockwise orientation and encloses the
poles $p\in\mathcal P^+$
while the contour $\mathcal C_T^-$ has anti-clockwise orientation and
encloses the poles $p\in\mathcal P^-$.}\label{fig:contour}
\end{figure}

%
%

\begin{Lemm}\label{lem2a}
Assume that all the elements of the sets $\mathcal P^+$ and $\mathcal
P^-$ are distinct.
Then, for any $T>0$ sufficiently large, and any $q, \theta\in\mbb
R^+\setminus\{0\}$
and $\lambda\in(0,\lambda^*]$ we have
%
\begin{equation}
\label{eq:keyid3} I^+_o(T):= \oint_{\mathcal C_T^+} f =
\frac{q}{q+\lambda} \WH\mu _\lambda (\theta),
\end{equation}
where $f$ is given in \eqref{eq:f}. Furthermore, we have
%
\begin{equation}
\label{eq:BW} \frac{1}{2\pi\ii}\int_{a-\ii\infty}^{a+\ii\infty}f(u)
\,\td u = \frac{q}{q+\lambda} \WH\mu_\lambda(\theta),\qquad  a\in\bigl(0,\bar\phi
(-\lambda)\bigr).
\end{equation}
In particular, for any $\lambda\in(0,\lambda^*]$,
$\mu_\lambda$ satisfies the identity in \eqref{eq:idd2}.
\end{Lemm}

\begin{Remm}
By arguments that are analogous, the ones
given below in the proof of Lemma~\ref{lem2a}, it may be verified that
the identity in \eqref{eq:keyid3} remains valid if
one replaces $\mathcal C_T^+$ by the contour $\mathcal C^-_T$
consisting of the segment $\mathcal I_T$ and the semi-circle
that joins $a-\ii T$ and $a+\ii T$ such that $\mathcal C^-_T$ is contained
in the set $\{u\in\mbb C\dvtx\Re(u)\leq a\}$
(see Figure~\ref{fig:contour}).
\end{Remm}
\begin{pf*}{Proof of Lemma~\ref{lem2a}}
By Cauchy's residue theorem,\vspace*{1pt}
the integral $I^+_o(T)$ of the function $f$ over the curve
$\mathcal C^+_T$ is for all $T$ sufficiently large equal to
%
\begin{equation}
\label{eq:IoT} I^+_o(T) = \frac{1}{2\pi\ii} \sum
_{p\in\mathcal P^+} n\bigl(\mathcal C_T^+,p\bigr)\mathrm
{Res}_p(f),
\end{equation}
where $\mathrm {Res}_p(f)$ denotes the residue of the function $f$ at the
pole $p$
and, for any $p\in\mbb C$ and any curve $\Gamma\dvtx[0,2\pi]\to
\mbb C$,
$n(\Gamma, p)$ denotes the winding number of $\Gamma$ around $p$. Note
that we have $n(\mathcal C_T^+,p)= -1$ for any $p\in\mathcal P^+$ (see
Figure~\ref{fig:contour}).

Since by assumption the poles are all distinct,
the residues at the poles $p\in\mathcal P^+$ satisfy
%
\begin{equation}
\label{eq:resf} \mathrm {Res}_p(f) = 2\pi\ii\cdot\lim
_{s\to p} (s-p) f(s), \qquad p\in \mathcal P^+.
\end{equation}
Inserting the explicit form of $f$ into \eqref{eq:resf} we find by
straightforward
algebra
%
\begin{equation}
\label{eq:Res} \Psi^+(q,\ii\theta)^{-1} \frac{\mathrm {Res}_{p}(f)}{2\pi\ii} = -
A^+(p)\cdot\frac{p}{p + \theta}, \qquad p\in\mathcal P^+,
\end{equation}
with
\begin{equation}
A^+(p) = \Psi^-(q,-\ii p) \frac{\bar\phi(-\lambda)}
{\bar\phi(-\lambda)-p}\frac{\prod_{k=1}^{m^+}
( 1 - {p}/{\alpha^+_k}  )}%
{\prod_{k=0, k\neq j}^{m^+}  (1 -{p}/{(\rho_k^+(-\lambda
))}  )}.
\label{eq:Ap}
\end{equation}
By using these explicit expressions,
we next verify that the following key-identity holds true:
%
\begin{equation}
\label{eq:resid} \frac{1}{2\pi\ii} \sum_{p\in\mathcal P^+} (-1)
\cdot\frac
{\mathrm
{Res}_p(f)}{\Psi
^+(q,\ii\theta)} = R(\theta) := \frac{q}{q+\lambda} \frac{\WH\mu_\lambda(\theta)}{\Psi
^+(q,\ii\theta)}.
\end{equation}
This identity follows from \eqref{eq:Res}
and the following partial-fraction decomposition of $R(\theta)$:
%
\begin{eqnarray}
\label{eq:R} &&R(\theta) = \frac{q}{q+\lambda} \Biggl[\sum
_{j=0}^{m^+} A^+\bigl(\rho_j^+(-\lambda)
\bigr) \frac{\rho
_j^+(-\lambda
)}{\rho_j^+(-\lambda) + \theta}
\nonumber
\\[-8pt]
\\[-8pt]
\nonumber
&&\hspace*{87pt}{}+ A^+\bigl(\bar\phi(-\lambda)\bigr)\frac
{\bar\phi
(-\lambda)}{\bar\phi(-\lambda) + \theta}
\Biggr],
\end{eqnarray}
where the coefficients $A^+(\bar\phi(-\lambda))$ and $A^+(\rho
_j^+(-\lambda))$, $j=0,\ldots, m^+$
are given by~\eqref{eq:Ap}.

We next show in two steps that \eqref{eq:R} holds.

(a) As a first step, we record the relation
%
\begin{equation}\qquad
\label{eq:PPP} \Psi^-\bigl(q,-\ii\rho_j^+(-\lambda)\bigr) =
\frac{q}{q+\lambda}\Psi ^+\bigl(q,-\ii\rho _j^+(-\lambda)
\bigr)^{-1},\qquad \lambda\in(0,\lambda^*].
\end{equation}
To see why this holds true, note that, for any $q>0$ and $\lambda\in
(0,\lambda^*]$,
it follows by analytical extension
that the Wiener--Hopf identities in \eqref{eq:WH}
remains valid for any $\theta\in\mbb C$ except some
finite set [namely, the sets of roots $\rho$ of the equation
$\Psi(\rho)=q$]. Substituting $\theta\to-\ii p$ $(p\in\mathcal
P^+)$ in
\eqref{eq:WH}
and using that by definition $\Psi(-\ii\rho^+_j(-\lambda))=-\lambda
$ we obtain
the relation \eqref{eq:PPP}.

(b) Inserting the explicit forms of
$\Psi^+(q,\ii\theta)$ and $\WH\mu_\lambda(\theta)$ [given in
\eqref
{eq:Psipm} and~\eqref{eq:mulme}]
into \eqref{eq:resid},
we find
\[
R(\theta) = \frac{q}{q+\lambda} \cdot \frac{\bar\phi(-\lambda)}
{\bar\phi(-\lambda)+\theta} \prod
_{k=0}^{m^+} \frac{1 + {\theta}/{(\rho^+_k(q))}}%
{1 + {\theta}/{(\rho_k^+(-\lambda))}}.
\]
It is a matter of algebra to verify that $R(\theta)$ admits a
partial-fraction decomposition of the form \eqref{eq:R}
for some coefficients $A^+(\bar\phi(-\lambda))$ and $A^+(\rho
_j^+(-\lambda))$, $j=0,\ldots, m^+$.
Furthermore, by deploying the identity on the left-hand side of \eqref{eq:PPP},
it is easy to show that these coefficients are equal
to the expression given in \eqref{eq:Res}, so that \eqref{eq:R} is
established.

Combining \eqref{eq:IoT} and \eqref{eq:resid} shows that for all $T$
sufficiently large,
we have
\[
I^+_o(T) = \frac{q}{q+\lambda}\WH\mu_\lambda(\theta).
\]
Finally, we note that the integral $I^+_c(T)$
over the semi-circles only (i.e., over $\mathcal C_T^+\setminus
\mathcal I_T$)
tends to zero as $T\to\infty$, since the length of the semi-circles
$\mathcal C^+_c(T)$
is proportional to $T$
while we have the bound $\max_{u\in\mathcal C^+_c(T)}|f(u)|\leq C^+/T^2$
for some constant $C^+>0$. Thus, we conclude that $I^+_o(T)$ converges
to the right-hand side
of \eqref{eq:BW} as $T$ tends to infinity, and the proof is complete.
\end{pf*}
\end{appendix}
\section*{Acknowledgements}
We thank two anonymous referees for their careful reading of the paper
and very constructive comments.
%





\printaddresses
\end{document}